%% file: CRFDR_article.tex
\documentclass[hidelinks,onefignum,onetabnum]{siamart220329}

\input{CRFDR_shared}

\ifpdf
\hypersetup{
  pdftitle={Computing Bouligand stationary points efficiently in low-rank optimization},
  pdfauthor={G. Olikier and P.-A. Absil}
}
\fi

\begin{document}

\maketitle

\begin{abstract}
This paper considers the problem of minimizing a differentiable function with locally Lipschitz continuous gradient on the algebraic variety of all $m$-by-$n$ real matrices of rank at most~$r$. Several definitions of stationarity exist for this nonconvex problem. Among them, Bouligand stationarity is the strongest necessary condition for local optimality. Only a handful of algorithms generate a sequence in the variety whose accumulation points are provably Bouligand stationary. Among them, the most parsimonious with (truncated) singular value decompositions (SVDs) or eigenvalue decompositions can still require a truncated SVD of a matrix whose rank can be as large as $\min\{m, n\}-r+1$ if the gradient does not have low rank, which is computationally prohibitive in the typical case where $r \ll \min\{m, n\}$. This paper proposes a first-order algorithm that generates a sequence in the variety whose accumulation points are Bouligand stationary while requiring SVDs of matrices whose smaller dimension is always at most $r$. A standard measure of Bouligand stationarity converges to zero along the bounded subsequences at a rate at least $O(1/\sqrt{i+1})$, where $i$ is the iteration counter. Furthermore, a rank-increasing scheme based on the proposed algorithm is presented, which can be of interest if the parameter $r$ is potentially overestimated.
\end{abstract}



\section{Introduction}
\label{sec:Introduction}
Low-rank optimization concerns the problem of minimizing a real-valued function over a space of matrices subject to an upper bound on their rank. Applications abound in various fields of science and engineering. The present work addresses specifically the problem
\begin{equation}
\label{eq:OptiDeterminantalVariety}
\min_{X \in \R_{\le r}^{m \times n}} f(X)
\end{equation}
of minimizing a differentiable function $f : \R^{m \times n} \to \R$ with locally Lipschitz continuous gradient on the determinantal variety
\begin{equation}
\label{eq:RealDeterminantalVariety}
\R_{\le r}^{m \times n} \coloneq \{X \in \R^{m \times n} \mid \rank X \le r\},
\end{equation}
$m$, $n$, and $r$ being positive integers such that $r < \min\{m, n\}$. The following basic facts about this problem are expounded in \cite[\S 1]{OlikierGallivanAbsil2024}: (i) determinantal varieties have been well studied in algebraic geometry; (ii) several instances of~\eqref{eq:OptiDeterminantalVariety} are important in many applications; (iii) in general, computing an exact or approximate minimizer of~\eqref{eq:OptiDeterminantalVariety} is intractable \cite{GillisGlineur2011}, and algorithms are only expected to find a stationary point of~\eqref{eq:OptiDeterminantalVariety} \cite{SchneiderUschmajew2015,ThemelisStellaPatrinos,HaLiuBarber2020,UschmajewVandereycken2020IMA,LiLuo2023,OlikierUschmajewVandereycken,LevinKileelBoumal2024,Pauwels,RebjockBoumal}, preferably a B-stationary point because B-stationarity is the strongest necessary condition for local optimality; (iv) $\pgd$ \cite[Algorithm~4.2]{OlikierWaldspurger}, $\ppgdr$ \cite[Algorithm~6.2]{OlikierGallivanAbsil2024}, $\rfdr$ \cite[Algorithm~3]{OlikierAbsil2023}, $\hrtr$ \cite[Algorithm~1]{LevinKileelBoumal2023}, and some of their hybridizations are the only algorithms known to generate a sequence in $\R_{\le r}^{m \times n}$ whose accumulation points are B-stationary for~\eqref{eq:OptiDeterminantalVariety}; (v) among these algorithms, $\rfdr$ is the most parsimonious with (truncated) singular value decompositions (SVDs) or eigenvalue decompositions \cite[Table~8.1]{OlikierGallivanAbsil2024}, though it can require a truncated SVD of a matrix whose rank can be as large as $\min\{m, n\}-r+1$ if the gradient does not have low rank, which is computationally prohibitive in the typical case where $r \ll \min\{m, n\}$.

This paper extends $\rfdr$ to a new class of methods, dubbed $\erfdr$, where ``E'' stands for ``extended'', which generate a sequence in $\R_{\le r}^{m \times n}$ whose accumulation points are B-stationary for~\eqref{eq:OptiDeterminantalVariety}. This class includes a subclass, called $\crfdr$, where ``C'' stands for ``cheap'', that involves SVDs of matrices whose smaller dimension is always at most $r$. A standard measure of B-stationarity converges to zero along bounded subsequences at a rate at least $O(1/\sqrt{i+1})$, where $i$ is the iteration counter. The same rate was achieved by \cite{OlikierUschmajewVandereycken}, but without guaranteed accumulation at B-stationary points of~\eqref{eq:OptiDeterminantalVariety}. Furthermore, a general rank-increasing algorithm based on $\erfdr$ is proposed, which can be of interest if the parameter $r$ in~\eqref{eq:OptiDeterminantalVariety} is potentially overestimated.

Here is a brief explanation of how $\rfdr$ was modified into $\crfdr$. Both share the same architecture: combining backtracking line searches and a rank reduction mechanism. They differ in the choice of the search direction, using respectively the $\rfd$ map \cite[Algorithm~1]{OlikierAbsil2023} and the $\crfd$ map (Definition~\ref{def:CRFD(R)map}) for their backtracking line searches. Given an input $X \in \R_{\le r}^{m \times n}$, the $\rfd$ map chooses its search direction to be an arbitrary projection of $-\nabla f(X)$ onto the restricted tangent cone to $\R_{\le r}^{m \times n}$ at $X$, which can be computed by \eqref{eq:ProjectionRestrictedTangentConeDeterminantalVariety} with $Z \coloneq -\nabla f(X)$. The above-mentioned computationally-prohibitive truncated SVDs are required if the rank $\ushort{r}$ of $X$ is smaller than~$r$; they are due to the second term of~\eqref{eq:ProjectionRestrictedTangentConeDeterminantalVariety}, which involves a projection onto $\R_{\le r-\ushort{r}}^{m \times n}$. In order to alleviate the computational burden due to this last projection, the proposed $\crfd$ map replaces the restricted tangent cone to $\R_{\le r}^{m \times n}$ at $X$ with a drastically simpler cone $\mathcal{C}$ of \emph{sparse} matrices of rank at most $r-\ushort{r}$, such as those mentioned in Table~\ref{tab:ClosedConeExamples}. It is readily seen that, for every $G \in \proj{\mathcal{C}}{-\nabla f(X)}$, the plain sum $X+G$ is of rank at most $r$; the \emph{retraction-free} nature of the $\rfd$ map is thus preserved. Moreover, this $G$ is a sufficient-descent direction in the sense of~\eqref{eq:ERFDmapSearchDirection} (Proposition~\ref{prop:CRFD}). With these two properties, the above-mentioned architecture yields a \emph{sufficient-descent map} (Lemma~\ref{lemma:ERFDRmap}) by an argument that was already used in the analysis of $\rfdr$~\cite[\S 6]{OlikierAbsil2023}. The convergence to the set of B-stationary points of~\eqref{eq:OptiDeterminantalVariety} follows right away (Theorem~\ref{thm:ERFDR}).

This paper is organized as follows. Preliminaries are covered in Section~\ref{sec:Preliminaries}. The $\erfd$ map, the $\erfdr$ algorithm, and a rank-increasing algorithm based on $\erfdr$ are defined and analyzed in Sections~\ref{sec:ERFDmap}, \ref{sec:ERFDR}, and~\ref{sec:RankIncreasingAlgorithm}, respectively. The $\crfd$ map and the $\crfdr$ algorithm are defined in Section~\ref{sec:CRFDmapCRFDR}. Their computational cost is analyzed in Section~\ref{sec:PracticalImplementationComputationalCostCRFDRmap}. Conclusions are drawn in Section~\ref{sec:Conclusion}.

\section{Preliminaries}
\label{sec:Preliminaries}
This section surveys the background material that is used throughout the paper; \cite{OlikierGallivanAbsil2024} and the references therein can be consulted for a more complete review of this material. Section~\ref{subsec:BasicFactsAboutSingularValues} recalls basic facts about singular values, and Section~\ref{subsec:ElementsVariationalGeometry} focuses on the background material from variational geometry.

The real vector space $\R^{m \times n}$ is endowed with the Frobenius inner product, defined by $\ip{X}{Y} \coloneq \tr Y^\tp X$, and $\norm{\cdot}$ denotes the Frobenius norm. For every $X \in \R^{m \times n}$ and $\rho \in (0, \infty)$, $\ball(X, \rho) \coloneq \{Y \in \R^{m \times n} \mid \norm{X-Y} < \rho\}$ and $\ball[X, \rho] \coloneq \{Y \in \R^{m \times n} \mid \norm{X-Y} \le \rho\}$ are respectively the open ball and the closed ball of center $X$ and radius $\rho$ in $\R^{m \times n}$. The zero matrix in $\R^{m \times n}$ is denoted by $0_{m \times n}$, and the identity matrix in $\R^{n \times n}$ is denoted by $I_n$. All operations on $\R$ or $\R^{m \times n}$ applied to subsets of $\R$ or $\R^{m \times n}$ are understood elementwise, and a singleton is identified with the element it contains.

Let $\mathcal{S}$ be a nonempty subset of $\R^{m \times n}$. The boundary of $\mathcal{S}$ is denoted by $\partial \mathcal{S}$. The distance from $X \in \R^{m \times n}$ to $\mathcal{S}$ is $\dist(X, \mathcal{S}) \coloneq \inf_{Y \in \mathcal{S}} \norm{X-Y}$, and the projection of $X$ onto $\mathcal{S}$ is $\proj{\mathcal{S}}{X} \coloneq \argmin_{Y \in \mathcal{S}} \norm{X-Y}$; the set $\proj{\mathcal{S}}{X}$ is nonempty and compact if $\mathcal{S}$ is closed \cite[Example~1.20]{RockafellarWets}. The set $\mathcal{S}$ is called a \emph{cone} if, for all $X \in \mathcal{S}$ and $\alpha \in [0, \infty)$, $\alpha X \in \mathcal{S}$. The \emph{polar} of a cone $\mathcal{C} \subseteq \R^{m \times n}$ is
\begin{equation*}
\mathcal{C}^* \coloneq \{Y \in \R^{m \times n} \mid \ip{X}{Y} \le 0 \; \forall X \in \mathcal{C}\},
\end{equation*}
which is a closed convex cone \cite[6(14)]{RockafellarWets}.

\begin{proposition}[{\cite[Proposition~A.6]{LevinKileelBoumal2023}}]
\label{prop:ProjectionOntoClosedCone}
Let $\mathcal{C} \subseteq \R^{m \times n}$ be a closed cone. For all $X \in \R^{m \times n}$ and $Y \in \proj{\mathcal{C}}{X}$,
\begin{equation*}
\ip{X}{Y} = \norm{Y}^2,
\end{equation*}
thus
\begin{equation*}
\norm{Y}^2 = \norm{X}^2 - \dist(X, \mathcal{C})^2,
\end{equation*}
and $\proj{\mathcal{C}}{X} = \{0_{m \times n}\}$ if and only if $X \in \mathcal{C}^*$.
\end{proposition}

\subsection{Basic facts about singular values}
\label{subsec:BasicFactsAboutSingularValues}
For every $\ushort{r} \in \{0, \dots, n\}$,
\begin{equation*}
\st(\ushort{r}, n) \coloneq \{U \in \R^{n \times \ushort{r}} \mid U^\tp U = I_{\ushort{r}}\}
\end{equation*}
is a Stiefel manifold \cite[\S 3.3.2]{AbsilMahonySepulchre}.
Let $X \in \R^{m \times n}$. The singular values of $X$ are denoted by $\sigma_1(X) \ge \dots \ge \sigma_{\min\{m, n\}}(X) \ge 0$. A full SVD of $X$ is a factorization $U \Sigma V^\tp$ of $X$, where $U \in \st(m, m)$, $\Sigma = \diag(\sigma_1(X), \dots, \sigma_{\ushort{r}}(X), 0_{m-\ushort{r} \times n-\ushort{r}})$, $V \in \st(n, n)$, and $\ushort{r} = \rank X$. (The $\diag$ operator returns a block diagonal matrix \cite[\S 1.3.1]{GolubVanLoan} whose diagonal blocks are the arguments given to the operator.) Given a full SVD $U \Sigma V^\tp$ of $X$, $s \in \{0, \dots, \min\{m, n\}-1\}$, and $\ushort{s} \coloneq \min\{\ushort{r}, s\}$, the factorizations $U(\mathord{:}, 1\mathord{:}\ushort{r}) \Sigma(1\mathord{:}\ushort{r}, 1\mathord{:}\ushort{r}) V(\mathord{:}, 1\mathord{:}\ushort{r})^\tp$ and $U(\mathord{:}, 1\mathord{:}\ushort{s}) \Sigma(1\mathord{:}\ushort{s}, 1\mathord{:}\ushort{s}) V(\mathord{:}, 1\mathord{:}\ushort{s})^\tp$ are respectively called an SVD of $X$ and a truncated SVD of rank at most $s$ of $X$. By the Eckart--Young theorem \cite{EckartYoung1936}, $\proj{\R_{\le s}^{m \times n}}{X}$ is the set of all truncated SVDs of rank at most $s$ of $X$.

Proposition~\ref{prop:SingularValuesLipschitz} states that the singular values are Lipschitz continuous with unit Lipschitz constant, a property that is used in the proof of Lemma~\ref{lemma:ERFDRmap}.

\begin{proposition}[{\cite[Corollary~8.6.2]{GolubVanLoan}}]
\label{prop:SingularValuesLipschitz}
For every $X, Y \in \R^{m \times n}$ and $j \in \{1, \dots, \min\{m, n\}\}$,
\begin{equation*}
|\sigma_j(X)-\sigma_j(Y)| \le \sigma_1(X-Y) \le \norm{X-Y}.
\end{equation*}
\end{proposition}

\subsection{Elements of variational geometry}
\label{subsec:ElementsVariationalGeometry}
Let $\mathcal{S}$ be a nonempty subset of $\R^{m \times n}$. A matrix $Z \in \R^{m \times n}$ is said to be \emph{tangent} to $\mathcal{S}$ at $X \in \mathcal{S}$ if there exist sequences $(X_i)_{i \in \N}$ in $\mathcal{S}$ converging to $X$ and $(t_i)_{i \in \N}$ in $(0, \infty)$ such that the sequence $(\frac{X_i-X}{t_i})_{i \in \N}$ converges to $Z$ \cite[Definition~6.1]{RockafellarWets}. The set of all tangents to $\mathcal{S}$ at $X \in \mathcal{S}$ is a closed cone \cite[Proposition~6.2]{RockafellarWets} called the \emph{(Bouligand) tangent cone} to $\mathcal{S}$ at $X$ and denoted by $\tancone{\mathcal{S}}{X}$.

\begin{definition}
\label{def:B-Stationarity}
A point $X \in \R_{\le r}^{m \times n}$ is said to be \emph{Bouligand stationary (B-stationary)} for~\eqref{eq:OptiDeterminantalVariety} if $\ip{\nabla f(X)}{Z} \ge 0$ for all $Z \in \tancone{\R_{\le r}^{m \times n}}{X}$.
\end{definition}

By Proposition~\ref{prop:ProjectionOntoClosedCone}, the function
\begin{equation}
\label{eq:NormProjectionNegativeGradientOntoTangentConeDeterminantalVariety}
\s(\cdot; f, \R_{\le r}^{m \times n}) : \R_{\le r}^{m \times n} \to \R : X \mapsto \norm{\proj{\tancone{\R_{\le r}^{m \times n}}{X}}{-\nabla f(X)}}
\end{equation}
is well defined. It is called a measure of B-stationarity for~\eqref{eq:OptiDeterminantalVariety} because a point in $\R_{\le r}^{m \times n}$ is B-stationary for~\eqref{eq:OptiDeterminantalVariety} if and only if it is a zero of this function; see \cite[\S 2.1]{SchneiderUschmajew2015} or \cite[Proposition~2.5]{LevinKileelBoumal2023}. This function is continuous on the smooth part
\begin{equation*}
\R_r^{m \times n} \coloneq \{X \in \R^{m \times n} \mid \rank X = r\}
\end{equation*}
of $\R_{\le r}^{m \times n}$ but can fail to be lower semicontinuous at every point of the singular part
\begin{equation*}
\R_{< r}^{m \times n} \coloneq \{X \in \R^{m \times n} \mid \rank X < r\}
\end{equation*}
of $\R_{\le r}^{m \times n}$ \cite[\S 1.1]{OlikierGallivanAbsil2024} because $\R_{\le r}^{m \times n}$ is not Clarke regular on $\R_{< r}^{m \times n}$ \cite[Remark~4]{SchneiderUschmajew2015}.

Given an embedded submanifold $\mathcal{M}$ of $\R^{m \times n}$ \cite[\S 3.3.1]{AbsilMahonySepulchre}, the tangent and normal spaces to $\mathcal{M}$ at $X \in \mathcal{M}$, which are the orthogonal complements of each other, are respectively denoted by $\tancone{\mathcal{M}}{X}$ and $\norcone{\mathcal{M}}{X}$; this is consistent since, by \cite[Example~6.8]{RockafellarWets}, the tangent cone to $\mathcal{M}$ at one of its points coincides with the tangent space to $\mathcal{M}$ at that point.

\begin{definition}
\label{def:RestrictedTangentCone}
Given $\mu \in (0, 1]$, a \emph{restricted tangent cone} to $\mathcal{S}$ is a map that associates with every $X \in \mathcal{S}$ a closed cone $\restancone{\mathcal{S}}{X}$ such that $\restancone{\mathcal{S}}{X} \subseteq \tancone{\mathcal{S}}{X}$, $X+\restancone{\mathcal{S}}{X} \subseteq \mathcal{S}$, and, for all $Z \in \R^{m \times n}$, $\norm{\proj{\restancone{\mathcal{S}}{X}}{Z}} \ge \mu \norm{\proj{\tancone{\mathcal{S}}{X}}{Z}}$.
\end{definition}

The rest of this section reviews descriptions of the tangent cone and restricted tangent cone to $\R_{\le r}^{m \times n}$ at $X \in \R_{\le r}^{m \times n}$, and the projections onto them, based on the Moore--Penrose inverse $X^\mpinv$; equivalent descriptions based on block matrices can be found, e.g., in \cite[\S 3]{OlikierAbsil2023}. The practical implications of these descriptions for the $\rfd$ map \cite[Algorithm~1]{OlikierAbsil2023} are also highlighted. Let $\ushort{r} \coloneq \rank X$.

By \cite[Proposition~4.1]{HelmkeShayman1995}, the tangent space to $\R_{\ushort{r}}^{m \times n}$ at $X$ is
\begin{equation}
\label{eq:TangentSpaceFixedRankManifold}
\tancone{\R_{\ushort{r}}^{m \times n}}{X}
= \R^{m \times m} X + X \R^{n \times n}
\end{equation}
and thus the normal space to $\R_{\ushort{r}}^{m \times n}$ at $X$ is
\begin{equation}
\label{eq:NormalSpaceFixedRankManifold}
\norcone{\R_{\ushort{r}}^{m \times n}}{X}
= (I_m-XX^\mpinv) \R^{m \times n} (I_n-X^\mpinv X).
\end{equation}
By \cite[Theorem~3.2]{SchneiderUschmajew2015}, the tangent cone to $\R_{\le r}^{m \times n}$ at $X$ is
\begin{equation}
\label{eq:TangentConeDeterminantalVariety}
\tancone{\R_{\le r}^{m \times n}}{X}
= \tancone{\R_{\ushort{r}}^{m \times n}}{X} + \left(\norcone{\R_{\ushort{r}}^{m \times n}}{X} \cap \R_{\le r-\ushort{r}}^{m \times n}\right).
\end{equation}
By \cite[Definition~3.1]{OlikierAbsil2023}, the restricted tangent cone to $\R_{\ushort{r}}^{m \times n}$ at $X$ is
\begin{equation}
\label{eq:RestrictedTangentConeFixedRankManifold}
\restancone{\R_{\ushort{r}}^{m \times n}}{X}
= \R^{m \times m} X \cup X \R^{n \times n}
\end{equation}
and the restricted tangent cone to $\R_{\le r}^{m \times n}$ at $X$ is
\begin{equation}
\label{eq:RestrictedTangentConeDeterminantalVariety}
\restancone{\R_{\le r}^{m \times n}}{X}
= \restancone{\R_{\ushort{r}}^{m \times n}}{X} + \left(\norcone{\R_{\ushort{r}}^{m \times n}}{X} \cap \R_{\le r-\ushort{r}}^{m \times n}\right).
\end{equation}
Crucially, by subadditivity of the rank, \eqref{eq:TangentSpaceFixedRankManifold} implies
\begin{equation}
\label{eq:MovingInTanSpaceFixedRankManifoldAtMostDoublesRank}
X+\tancone{\R_{\ushort{r}}^{m \times n}}{X} \subseteq \R_{\le 2\ushort{r}}^{m \times n},
\end{equation}
and \eqref{eq:RestrictedTangentConeFixedRankManifold} implies
\begin{equation}
\label{eq:MovingInResTanConeFixedRankManifoldPreservesFeasibility}
X+\restancone{\R_{\ushort{r}}^{m \times n}}{X} \subseteq \R_{\le \ushort{r}}^{m \times n}.
\end{equation}
By subadditivity of the rank, \eqref{eq:TangentConeDeterminantalVariety} and \eqref{eq:MovingInTanSpaceFixedRankManifoldAtMostDoublesRank} yield
\begin{equation}
\label{eq:MovingInTanConeDetVarAtMostDoublesRank}
X+\tancone{\R_{\le r}^{m \times n}}{X} \subseteq \R_{\le r+\ushort{r}}^{m \times n},
\end{equation}
and \eqref{eq:RestrictedTangentConeDeterminantalVariety} and \eqref{eq:MovingInResTanConeFixedRankManifoldPreservesFeasibility} yield
\begin{equation}
\label{eq:MovingInResTanConePreservesFeasibility}
X+\restancone{\R_{\le r}^{m \times n}}{X} \subseteq \R_{\le r}^{m \times n}.
\end{equation}
By \cite[Theorem~3.1 and Corollary~3.3]{SchneiderUschmajew2015} and \cite[Proposition~3.2]{OlikierAbsil2023}, for all $Z \in \R^{m \times n}$,
\begin{align}
\label{eq:ProjectionTangentSpaceFixedRankManifold}
\proj{\tancone{\R_{\ushort{r}}^{m \times n}}{X}}{Z}
&= X X^\mpinv Z + Z X^\mpinv X - X X^\mpinv Z X^\mpinv X,\\
\label{eq:ProjectionNormalSpaceFixedRankManifold}
\proj{\norcone{\R_{\ushort{r}}^{m \times n}}{X}}{Z}
&= Z - \proj{\tancone{\R_{\ushort{r}}^{m \times n}}{X}}{Z}
= (I_m-XX^\mpinv)Z(I_n-X^\mpinv X),\\
\label{eq:ProjectionRestrictedTangentConeFixedRankManifold}
\proj{\restancone{\R_{\ushort{r}}^{m \times n}}{X}}{Z}
&= \left\{\begin{array}{ll}
X X^\mpinv Z & \text{if } \norm{X X^\mpinv Z} > \norm{Z X^\mpinv X},\\
\{X X^\mpinv Z, Z X^\mpinv X\} & \text{if } \norm{X X^\mpinv Z} = \norm{Z X^\mpinv X},\\
Z X^\mpinv X & \text{if } \norm{X X^\mpinv Z} < \norm{Z X^\mpinv X},
\end{array}\right.\\
\label{eq:ProjectionTangentConeDeterminantalVariety}
\proj{\tancone{\R_{\le r}^{m \times n}}{X}}{Z}
&= \proj{\tancone{\R_{\ushort{r}}^{m \times n}}{X}}{Z} + \proj{\R_{\le r-\ushort{r}}^{m \times n}}{\proj{\norcone{\R_{\ushort{r}}^{m \times n}}{X}}{Z}},\\
\label{eq:ProjectionRestrictedTangentConeDeterminantalVariety}
\proj{\restancone{\R_{\le r}^{m \times n}}{X}}{Z}
&= \proj{\restancone{\R_{\ushort{r}}^{m \times n}}{X}}{Z} + \proj{\R_{\le r-\ushort{r}}^{m \times n}}{\proj{\norcone{\R_{\ushort{r}}^{m \times n}}{X}}{Z}},\\
\label{eq:NormProjectionRestrictedTangentConeDeterminantalVariety}
\norm{\proj{\tancone{\R_{\le r}^{m \times n}}{X}}{Z}}
&\ge \norm{\proj{\restancone{\R_{\le r}^{m \times n}}{X}}{Z}}
\ge \frac{1}{\sqrt{2}} \norm{\proj{\tancone{\R_{\le r}^{m \times n}}{X}}{Z}},\\
\label{eq:NormProjectionTangentConeDeterminantalVariety}
\norm{Z}
&\ge \norm{\proj{\tancone{\R_{\le r}^{m \times n}}{X}}{Z}}
\ge \sqrt{\frac{r-\ushort{r}}{\min\{m,n\}-\ushort{r}}} \norm{Z}.
\end{align}

As indicated in Section~\ref{sec:Introduction}, the $\rfd$ map performs a backtracking line search along an arbitrary element of $\proj{\restancone{\R_{\le r}^{m \times n}}{X}}{-\nabla f(X)}$. If $\ushort{r} = r$, \eqref{eq:ProjectionRestrictedTangentConeDeterminantalVariety} reduces to \eqref{eq:ProjectionRestrictedTangentConeFixedRankManifold} because its second term is $0_{m \times n}$. Computing \eqref{eq:ProjectionRestrictedTangentConeFixedRankManifold} merely requires orthonormal bases of the row and column spaces of $X$ and matrix multiplications. If $\ushort{r} < r$, then computing the second term of~\eqref{eq:ProjectionRestrictedTangentConeDeterminantalVariety} requires a truncated SVD of rank at most $r-\ushort{r}$ of a matrix whose rank can be as large as $\min\{m, n\}-\ushort{r}$ by~\eqref{eq:ProjectionNormalSpaceFixedRankManifold}. This is computationally prohibitive in the typical case where $r \ll \min\{m, n\}$, and thus raises the question of the possibility of finding a computationally cheaper surrogate for~\eqref{eq:ProjectionRestrictedTangentConeDeterminantalVariety} that preserves the properties of the $\rfd$ map. The question is answered positively in Section~\ref{sec:CRFDmapCRFDR} based on the analysis conducted in Section~\ref{sec:ERFDmap}.

\section{The $\erfd$ map}
\label{sec:ERFDmap}
In this section, the $\erfd$ map is defined (Algorithm~\ref{algo:ERFDmap}) and analyzed (Corollary~\ref{coro:ERFDmapArmijoCondition}). This serves as a basis for the definition and analysis of the $\erfdr$ algorithm in Section~\ref{sec:ERFDR} since the latter uses the $\erfd$ map as a subroutine.

Given parameters $\oshort{\alpha} \in (0, \infty)$, $\kappa_1 \in (0, \frac{1}{2}]$, and $\kappa_2 \in (0, 1]$ and an input $X \in \R_{\le r}^{m \times n}$, the $\erfd$ map chooses $G \in \R^{m \times n}$ such that
\begin{equation}
\label{eq:ERFDmapSearchDirection}
\ip{G}{-\nabla f(X)} \ge \max\{\kappa_1 \s(X; f, \R_{\le r}^{m \times n})^2, \kappa_2 \norm{G}^2\}
\end{equation}
and $X+\alpha G \in \R_{\le r}^{m \times n}$ for all $\alpha \in [0, \oshort{\alpha}]$, then performs a backtracking line search along $G$: it computes $f(X+\alpha G)$ for decreasing values of the step size $\alpha \in (0, \oshort{\alpha}]$ and outputs $X+\alpha G$ as soon as an Armijo condition is satisfied. Thus, the $\erfd$ map is retraction-free: it updates its input by moving along a straight line.

The inequality~\eqref{eq:ERFDmapSearchDirection} means that the decrease in $f$ along $G$ at $X$ is at least a fraction of that along an arbitrary element of $\proj{\tancone{\R_{\le r}^{m \times n}}{X}}{-\nabla f(X)}$, and that this cannot be achieved by choosing $G$ too large. It is related to \cite[(2)--(3)]{GrippoLamparielloLucidi} by the Cauchy--Schwarz inequality. Specifically, \cite[(2)--(3)]{GrippoLamparielloLucidi} implies \eqref{eq:ERFDmapSearchDirection} with $\kappa_1 = c_1$ and $\kappa_2 = \frac{c_1}{c_2^2}$, and \eqref{eq:ERFDmapSearchDirection} with $\s(X; f, \R_{\le r}^{m \times n})$ replaced with $\norm{\nabla f(X)}$ implies \cite[(2)--(3)]{GrippoLamparielloLucidi} with $c_1 = \kappa_1$ and $c_2 = \frac{1}{\kappa_2}$. The inequality~\eqref{eq:ERFDmapSearchDirection} also implies the angle condition \cite[(2.15)]{SchneiderUschmajew2015} with $\omega = \sqrt{\kappa_1\kappa_2}$.

If $\s(X; f, \R_{\le r}^{m \times n}) = 0$, then the set of all $G \in \R^{m \times n}$ satisfying \eqref{eq:ERFDmapSearchDirection} and $X+\alpha G \in \R_{\le r}^{m \times n}$ for all $\alpha \in [0, \oshort{\alpha}]$ is $\{0_{m \times n}\}$. In practice, $\s(X; f, \R_{\le r}^{m \times n})$ should not be computed if $\rank X < r$; in that case, by the second inequality of~\eqref{eq:NormProjectionTangentConeDeterminantalVariety}, $\s(X; f, \R_{\le r}^{m \times n}) = 0$ if and only if $\norm{\nabla f(X)} = 0$.

By~\eqref{eq:MovingInResTanConePreservesFeasibility}, if $G \in \restancone{\R_{\le r}^{m \times n}}{X}$, then $X+\alpha G \in \R_{\le r}^{m \times n}$ for all $\alpha \in (0, \infty)$. Moreover, the $\erfd$ map extends the $\rfd$ map \cite[Algorithm~1]{OlikierAbsil2023} because every $G \in \proj{\restancone{\R_{\le r}^{m \times n}}{X}}{-\nabla f(X)}$ satisfies \eqref{eq:ERFDmapSearchDirection} with $\kappa_1 = \frac{1}{2}$ and $\kappa_2 = 1$. Hence, the subset of $\R^{m \times n}$ in which the $\erfd$ map chooses an element is not empty.

\begin{algorithm}[H]
\caption{$\erfd$ map on $\R_{\le r}^{m \times n}$}
\label{algo:ERFDmap}
\begin{algorithmic}[1]
\Require
$(f, r, \ushort{\alpha}, \oshort{\alpha}, \beta, c, \kappa_1, \kappa_2)$ where $f : \R^{m \times n} \to \R$ is differentiable with $\nabla f$ locally Lipschitz continuous, $r < \min\{m, n\}$ is a positive integer, $0 < \ushort{\alpha} \le \oshort{\alpha} < \infty$, $\beta, c \in (0, 1)$, $\kappa_1 \in (0, \frac{1}{2}]$, and $\kappa_2 \in (0, 1]$.
\Input
$X \in \R_{\le r}^{m \times n}$.
\Output
a point in $\erfd(X; f, r, \ushort{\alpha}, \oshort{\alpha}, \beta, c, \kappa_1, \kappa_2)$.

\State
Choose $G \in \R^{m \times n}$ satisfying \eqref{eq:ERFDmapSearchDirection} and $X+\alpha G \in \R_{\le r}^{m \times n}$ for all $\alpha \in [0, \oshort{\alpha}]$;
\label{algo:ERFDmap:SearchDirection}
\State
Choose $\alpha \in [\ushort{\alpha}, \oshort{\alpha}]$;
\label{algo:ERFDmap:InitialStepSize}
\While
{$f(X + \alpha G) > f(X) + c \, \alpha \, \ip{\nabla f(X)}{G}$}
\State
$\alpha \gets \alpha \beta$;
\EndWhile
\State
Return $X + \alpha G$.
\end{algorithmic}
\end{algorithm}

By \cite[Lemma~1.2.3]{Nesterov2018}, since $\nabla f$ is locally Lipschitz continuous, i.e., for every closed ball $\mathcal{B} \subsetneq \R^{m \times n}$,
\begin{equation*}
\lip_{\mathcal{B}}(\nabla f) \coloneq \sup_{\substack{X, Y \in \mathcal{B} \\ X \ne Y}} \frac{\norm{\nabla f(X) - \nabla f(Y)}}{\norm{X-Y}} < \infty,
\end{equation*}
it holds for all $X, Y \in \mathcal{B}$ that
\begin{equation}
\label{eq:InequalityLipschitzContinuousGradient}
|f(Y) - f(X) - \ip{\nabla f(X)}{Y-X}| \le \frac{\lip_{\mathcal{B}}(\nabla f)}{2} \norm{Y-X}^2.
\end{equation}

Proposition~\ref{prop:ERFDmapUpperBoundCost} and Corollary~\ref{coro:ERFDmapArmijoCondition} respectively extend \cite[Proposition~4.1 and Corollary~4.2]{OlikierAbsil2023}.

\begin{proposition}
\label{prop:ERFDmapUpperBoundCost}
Let $\oshort{\alpha} \in (0, \infty)$, $\kappa_1 \in (0, \frac{1}{2}]$, $\kappa_2 \in (0, 1]$, and $X \in \R_{\le r}^{m \times n}$. Let $\mathcal{B} \subsetneq \R^{m \times n}$ be a closed ball containing $\ball[X, \frac{\oshort{\alpha}}{\kappa_2} \norm{\nabla f(X)}]$ if $\norm{\nabla f(X)} \ne 0$. Let $G \in \R^{m \times n}$ satisfy \eqref{eq:ERFDmapSearchDirection} and $X+\alpha G \in \R_{\le r}^{m \times n}$ for all $\alpha \in [0, \oshort{\alpha}]$. Then, for all $\alpha \in [0, \oshort{\alpha}]$,
\begin{equation}
\label{eq:ERFDmapUpperBoundCost}
f(X + \alpha G) \le f(X) + \alpha \left(1-\frac{\alpha}{2\kappa_2}\lip_\mathcal{B}(\nabla f)\right) \ip{\nabla f(X)}{G}.
\end{equation}
\end{proposition}

\begin{proof}
Since $\norm{G}^2 \le \frac{1}{\kappa_2} \ip{G}{-\nabla f(X)} \le \frac{1}{\kappa_2} \norm{G} \norm{\nabla f(X)}$, it holds that $\norm{G} \le \frac{1}{\kappa_2} \norm{\nabla f(X)}$. Thus, $X+\alpha G \in \ball[X, \frac{\oshort{\alpha}}{\kappa_2} \norm{\nabla f(X)}]$ for all $\alpha \in [0, \oshort{\alpha}]$. The inequality \eqref{eq:ERFDmapUpperBoundCost} is based on \eqref{eq:InequalityLipschitzContinuousGradient}:
\begin{align*}
f(X+\alpha G)-f(X)
&\le \ip{\nabla f(X)}{(X+\alpha G)-X} + \frac{\lip_\mathcal{B}(\nabla f)}{2} \norm{(X+\alpha G)-X}^2\\
&= \alpha \ip{\nabla f(X)}{G} + \frac{\alpha^2}{2} \lip_\mathcal{B}(\nabla f) \norm{G}^2\\
&\le \alpha \left(1-\frac{\alpha}{2\kappa_2}\lip_\mathcal{B}(\nabla f)\right) \ip{\nabla f(X)}{G}.
\end{align*}
\end{proof}

Corollary~\ref{coro:ERFDmapArmijoCondition} provides a lower bound on the step size of the $\erfd$ map. It plays an instrumental role in the proof of Lemma~\ref{lemma:ERFDRmap}.

\begin{corollary}
\label{coro:ERFDmapArmijoCondition}
Let $\mathcal{B}$ be a closed ball as in Proposition~\ref{prop:ERFDmapUpperBoundCost}. The while loop in Algorithm~\ref{algo:ERFDmap} terminates after at most
\begin{equation}
\label{eq:ERFDmapMaxNumIterationsWhile}
\max\left\{0, \left\lceil\ln\left(\frac{2\kappa_2(1-c)}{\alpha_0\lip_\mathcal{B}(\nabla f)}\right)/\ln(\beta)\right\rceil\right\}
\end{equation}
iterations, where $\alpha_0$ is the initial step size chosen in line~\ref{algo:ERFDmap:InitialStepSize}. Moreover, every $\tilde{X} \in \hyperref[algo:ERFDmap]{\erfd}(X; f, r, \ushort{\alpha}, \oshort{\alpha}, \beta, c, \kappa_1, \kappa_2)$ satisfies the Armijo condition
\begin{equation}
\label{eq:ERFDmapArmijoCondition}
f(\tilde{X}) \le f(X) + c \, \alpha \, \ip{\nabla f(X)}{G}
\end{equation}
with a step size $\alpha \in \left[\min\left\{\ushort{\alpha}, \frac{2\beta\kappa_2(1-c)}{\lip_\mathcal{B}(\nabla f)}\right\}, \oshort{\alpha}\right]$, which implies
\begin{equation}
\label{eq:ERFDmapCostDecrease}
f(\tilde{X}) \le f(X) - c \, \kappa_1 \, \alpha \s(X; f, \R_{\le r}^{m \times n})^2.
\end{equation}
\end{corollary}

\begin{proof}
If $\ip{\nabla f(X)}{G} = 0$, then $G = 0_{m \times n}$, hence the while loop is not executed and $\tilde{X} = X$. Assume that $\ip{\nabla f(X)}{G} \ne 0$. Define $\alpha_* \coloneq \frac{2\kappa_2(1-c)}{\lip_\mathcal{B}(\nabla f)}$. For all $\alpha \in (0, \infty)$,
\begin{equation*}
f(X) + \alpha \left(1-\frac{\alpha}{2\kappa_2}\lip_\mathcal{B}(\nabla f)\right) \ip{\nabla f(X)}{G}
\le f(X) + c \, \alpha \, \ip{\nabla f(X)}{G}
\quad\text{iff}\quad
\alpha \le \alpha_*.
\end{equation*}
Since the left-hand side of the first inequality is an upper bound on $f(X + \alpha G)$ for all $\alpha \in (0, \oshort{\alpha}]$, the Armijo condition \eqref{eq:ERFDmapArmijoCondition} is necessarily satisfied if $\alpha \in (0, \min\{\oshort{\alpha}, \alpha_*\}]$.
Therefore, at the latest, the while loop ends after iteration $i \in \N \setminus \{0\}$ with $\alpha = \alpha_0\beta^i$ such that $\frac{\alpha}{\beta} > \alpha_*$, hence $i < 1+\ln(\alpha_*/\alpha_0)/\ln(\beta)$, and thus $i \le \lceil\ln(\alpha_*/\alpha_0)/\ln(\beta)\rceil$.
\end{proof}

\section{The $\erfdr$ algorithm and its convergence analysis}
\label{sec:ERFDR}
The $\erfdr$ map, defined as Algorithm~\ref{algo:ERFDRmap}, is the iteration map of the $\erfdr$ algorithm. Given an input $X \in \R_{\le r}^{m \times n}$, the $\erfdr$ map proceeds as follows: (i) it applies the $\erfd$ map (Algorithm~\ref{algo:ERFDmap}) to $X$, thereby producing a point $\tilde{X}$, (ii) if $\sigma_r(X)$ is positive but smaller than some threshold $\Delta \in (0, \infty)$, it applies the $\erfd$ map to a projection $\hat{X}$ of $X$ onto $\R_{r-1}^{m \times n}$, then producing a point $\tilde{X}^\mathrm{R}$, and (iii) it outputs a point among $\tilde{X}$ and $\tilde{X}^\mathrm{R}$ that maximally decreases $f$.

\begin{algorithm}[H]
\caption{$\erfdr$ map on $\R_{\le r}^{m \times n}$}
\label{algo:ERFDRmap}
\begin{algorithmic}[1]
\Require
$(f, r, \ushort{\alpha}, \oshort{\alpha}, \beta, c, \kappa_1, \kappa_2, \Delta)$ where $f : \R^{m \times n} \to \R$ is differentiable with $\nabla f$ locally Lipschitz continuous, $r < \min\{m, n\}$ is a positive integer, $0 < \ushort{\alpha} \le \oshort{\alpha} < \infty$, $\beta, c \in (0, 1)$, $\kappa_1 \in (0, \frac{1}{2}]$, $\kappa_2 \in (0, 1]$, and $\Delta \in (0, \infty)$.
\Input
$X \in \R_{\le r}^{m \times n}$.
\Output
a point in $\erfdr(X; f, r, \ushort{\alpha}, \oshort{\alpha}, \beta, c, \kappa_1, \kappa_2, \Delta)$.

\State
Choose $\tilde{X} \in \hyperref[algo:ERFDmap]{\erfd}(X; f, r, \ushort{\alpha}, \oshort{\alpha}, \beta, c, \kappa_1, \kappa_2)$;
\label{algo:ERFDRmap:ERFD}
\If
{$\sigma_r(X) \in (0, \Delta]$}
\State
Choose $\hat{X} \in \proj{\R_{r-1}^{m \times n}}{X}$;
\State
Choose $\tilde{X}^\mathrm{R} \in \hyperref[algo:ERFDmap]{\erfd}(\hat{X}; f, r, \ushort{\alpha}, \oshort{\alpha}, \beta, c, \kappa_1, \kappa_2)$;
\label{algo:ERFDRmap:ERFD_RankReduction}
\State
Return $Y \in \argmin_{\{\tilde{X}, \tilde{X}^\mathrm{R}\}} f$.
\Else
\State
Return $\tilde{X}$.
\EndIf
\end{algorithmic}
\end{algorithm}

The $\erfdr$ algorithm is defined as Algorithm~\ref{algo:ERFDR}. It generates a sequence in $\R_{\le r}^{m \times n}$ along which $f$ is strictly decreasing.

\begin{algorithm}[H]
\caption{$\erfdr$ on $\R_{\le r}^{m \times n}$}
\label{algo:ERFDR}
\begin{algorithmic}[1]
\Require
$(f, r, \ushort{\alpha}, \oshort{\alpha}, \beta, c, \kappa_1, \kappa_2, \Delta)$ where $f : \R^{m \times n} \to \R$ is differentiable with $\nabla f$ locally Lipschitz continuous, $r < \min\{m, n\}$ is a positive integer, $0 < \ushort{\alpha} \le \oshort{\alpha} < \infty$, $\beta, c \in (0, 1)$, $\kappa_1 \in (0, \frac{1}{2}]$, $\kappa_2 \in (0, 1]$, and $\Delta \in (0, \infty)$.
\Input
$X_0 \in \R_{\le r}^{m \times n}$.
\Output
a sequence in $\R_{\le r}^{m \times n}$.

\State
$i \gets 0$;
\While
{$\s(X_i; f, \R_{\le r}^{m \times n}) > 0$}
\State
Choose $\Delta_i \in [\Delta, \infty)$;
\State
Choose $X_{i+1} \in \hyperref[algo:ERFDRmap]{\erfdr}(X_i; f, r, \ushort{\alpha}, \oshort{\alpha}, \beta, c, \kappa_1, \kappa_2, \Delta_i)$;
\State
$i \gets i+1$;
\EndWhile
\end{algorithmic}
\end{algorithm}

Algorithms~\ref{algo:ERFDRmap} and~\ref{algo:ERFDR} are analyzed in Section~\ref{subsec:ConvergenceAnalysis} by deploying the strategy laid out in~\cite[\S 3]{OlikierGallivanAbsil2024}. The crucial step is to prove that the $\erfdr$ map is an \emph{$(f, \R_{\le r}^{m \times n})$-sufficient-descent map} (Lemma~\ref{lemma:ERFDRmap}). The choice of the parameter $\Delta$ is briefly discussed in Section~\ref{subsec:ChoiceDelta}.

\subsection{Convergence analysis}
\label{subsec:ConvergenceAnalysis}
Lemma~\ref{lemma:ERFDRmap} states that, for every $\ushort{X} \in \R_{\le r}^{m \times n}$ that is not B-stationary for~\eqref{eq:OptiDeterminantalVariety}, the decrease in $f$ obtained by applying the $\erfdr$ map to every $X \in \R_{\le r}^{m \times n}$ sufficiently close to $\ushort{X}$ is bounded away from zero.

\begin{lemma}
\label{lemma:ERFDRmap}
The $\erfdr$ map (Algorithm~\ref{algo:ERFDRmap}) is an $(f, \R_{\le r}^{m \times n})$-sufficient-descent map in the sense of \cite[Definition~3.1]{OlikierGallivanAbsil2024}: for every $\ushort{X} \in \R_{\le r}^{m \times n}$ that is not B-stationary for~\eqref{eq:OptiDeterminantalVariety}, there exist $\varepsilon(\ushort{X}), \delta(\ushort{X}) \in (0,\infty)$ such that, for all $X \in \ball[\ushort{X}, \varepsilon(\ushort{X})] \cap \R_{\le r}^{m \times n}$ and $Y \in \hyperref[algo:ERFDRmap]{\erfdr}(X; f, r, \ushort{\alpha}, \oshort{\alpha}, \beta, c, \kappa_1, \kappa_2, \Delta)$,
\begin{equation}
\label{eq:SufficientDecrease}
f(Y) - f(X) \le - \delta(\ushort{X}).
\end{equation}
\end{lemma}

\begin{proof}
Let $\ushort{X} \in \R_{\le r}^{m \times n}$ be such that $\s(\ushort{X}; f, \R_{\le r}^{m \times n}) > 0$. Define $\ushort{r} \coloneq \rank \ushort{X}$. This proof uses the same arguments as that of \cite[Proposition~6.1]{OlikierAbsil2023}. It constructs $\varepsilon(\ushort{X})$ and $\delta(\ushort{X})$ based on~\eqref{eq:ERFDmapCostDecrease}. This requires a lower bound on $\s(\cdot; f, \R_{\le r}^{m \times n})$ holding in a suitable neighborhood of $\ushort{X}$. It first considers the case where $\ushort{r} = r$, in which the lower bound on $\s(\cdot; f, \R_{\le r}^{m \times n})$ follows from its continuity on $\ball(\ushort{X}, \sigma_r(\ushort{X})) \cap \R_{\le r}^{m \times n}$. Then, it focuses on the case where $\ushort{r} < r$, in which the lower bound on $\s(\cdot; f, \R_{\le r}^{m \times n})$ follows from the bounds \eqref{eq:ERFDRmapContinuityGradient} on $\nabla f$ by the second inequality of~\eqref{eq:NormProjectionTangentConeDeterminantalVariety}. If $\rank X < r$, then the second inequality of~\eqref{eq:NormProjectionTangentConeDeterminantalVariety} readily gives a lower bound on $\s(X; f, \R_{\le r}^{m \times n})$. This is not the case if $\rank X = r$, however. This is where the rank reduction mechanism comes into play. It considers a projection $\hat{X}$ of $X$ onto $\R_{r-1}^{m \times n}$, and the second inequality of~\eqref{eq:NormProjectionTangentConeDeterminantalVariety} gives a lower bound on $\s(\hat{X}; f, \R_{\le r}^{m \times n})$. The inequality \eqref{eq:SufficientDecrease} is then obtained from \eqref{eq:ERFDRmapCostContinuous}, which follows from the continuity of $f$ at $\ushort{X}$.

By the first inequality of~\eqref{eq:NormProjectionTangentConeDeterminantalVariety}, $\norm{\nabla f(\ushort{X})} \ge \s(\ushort{X}; f, \R_{\le r}^{m \times n})$. Since $\nabla f$ is continuous at $\ushort{X}$, there exists $\rho_1(\ushort{X}) \in (0, \infty)$ such that, for all $X \in \ball[\ushort{X}, \rho_1(\ushort{X})]$, $\norm{\nabla f(X)-\nabla f(\ushort{X})} \le \frac{1}{2} \norm{\nabla f(\ushort{X})}$ and hence, as $|\norm{\nabla f(X)}-\norm{\nabla f(\ushort{X})}| \le \norm{\nabla f(X)-\nabla f(\ushort{X})}$,
\begin{equation}
\label{eq:ERFDRmapContinuityGradient}
\frac{1}{2} \norm{\nabla f(\ushort{X})} \le \norm{\nabla f(X)} \le \frac{3}{2} \norm{\nabla f(\ushort{X})}.
\end{equation}
Define
\begin{align*}
\oshort{\rho}(\ushort{X}) \coloneq \rho_1(\ushort{X}) + \frac{3\oshort{\alpha}}{2\kappa_2} \norm{\nabla f(\ushort{X})},&&
\alpha_*(\ushort{X}) \coloneq \min\left\{\ushort{\alpha}, \frac{2\beta\kappa_2(1-c)}{\lip_{\ball[\ushort{X}, \bar{\rho}(\ushort{X})]}(\nabla f)}\right\}.
\end{align*}
Then, for every $X \in \ball[\ushort{X}, \rho_1(\ushort{X})]$, the inclusion $\ball[X, \frac{\oshort{\alpha}}{\kappa_2} \norm{\nabla f(X)}] \subseteq \ball[\ushort{X}, \bar{\rho}(\ushort{X})]$ holds since, for all $Z \in \ball[X, \frac{\oshort{\alpha}}{\kappa_2} \norm{\nabla f(X)}]$,
\begin{equation*}
\norm{Z-\ushort{X}}
\le \norm{Z-X} + \norm{X-\ushort{X}}
\le \frac{\oshort{\alpha}}{\kappa_2} \norm{\nabla f(X)} + \rho_1(\ushort{X})
\le \bar{\rho}(\ushort{X}),
\end{equation*}
where the last inequality follows from the second inequality of~\eqref{eq:ERFDRmapContinuityGradient}. Thus, the closed ball $\ball[\ushort{X}, \bar{\rho}(\ushort{X})]$ satisfies the condition from Proposition~\ref{prop:ERFDmapUpperBoundCost}. Hence, for all $X \in \ball[\ushort{X}, \rho_1(\ushort{X})] \cap \R_{\le r}^{m \times n}$ and $\tilde{X} \in \hyperref[algo:ERFDmap]{\erfd}(X; f, r, \ushort{\alpha}, \oshort{\alpha}, \beta, c, \kappa_1, \kappa_2)$, Corollary~\ref{coro:ERFDmapArmijoCondition} applies with this closed ball, and \eqref{eq:ERFDmapCostDecrease} yields
\begin{equation}
\label{eq:ERFDRmapERFDmapArmijoCondition}
f(\tilde{X}) \le f(X) - c \, \kappa_1 \alpha_*(\ushort{X}) \s(X; f, \R_{\le r}^{m \times n})^2.
\end{equation}
Define
\begin{equation}
\label{eq:ERFDRmapDelta}
\delta(\ushort{X}) \coloneq \left\{\begin{array}{ll}
\dfrac{c}{4} \kappa_1 \alpha_*(\ushort{X}) \s(\ushort{X}; f, \R_{\le r}^{m \times n})^2 & \text{if } \ushort{r} = r,\\
\dfrac{c \, \kappa_1 \alpha_*(\ushort{X}) \norm{\nabla f(\ushort{X})}^2}{12(\min\{m, n\}-r+1)} & \text{if } \ushort{r} < r.
\end{array}\right.
\end{equation}

Let us consider the case where $\ushort{r} = r$. On $\ball(\ushort{X}, \sigma_r(\ushort{X})) \cap \R_{\le r}^{m \times n} = \ball(\ushort{X}, \sigma_r(\ushort{X})) \cap \R_r^{m \times n}$, $\s(\cdot; f, \R_{\le r}^{m \times n})$ coincides with the norm of the Riemannian gradient of the restriction of $f$ to the smooth manifold $\R_r^{m \times n}$, which is continuous \cite[\S 2.1]{SchneiderUschmajew2015}. Therefore, there exists $\rho_2(\ushort{X}) \in (0,\sigma_r(\ushort{X}))$ such that $\s(\ball[\ushort{X}, \rho_2(\ushort{X})] \cap \R_r^{m \times n}; f, \R_{\le r}^{m \times n}) \subseteq [\frac{1}{2}\s(\ushort{X}; f, \R_{\le r}^{m \times n}), \frac{3}{2}\s(\ushort{X}; f, \R_{\le r}^{m \times n})]$. Define
\begin{equation*}
\varepsilon(\ushort{X}) \coloneq \min\{\rho_1(\ushort{X}), \rho_2(\ushort{X})\}.
\end{equation*}
Let $X \in \ball[\ushort{X}, \varepsilon(\ushort{X})] \cap \R_{\le r}^{m \times n}$ and $Y \in \hyperref[algo:ERFDRmap]{\erfdr}(X; f, r, \ushort{\alpha}, \oshort{\alpha}, \beta, c, \kappa_1, \kappa_2, \Delta)$. There exists $\tilde{X} \in \hyperref[algo:ERFDmap]{\erfd}(X; f, r, \ushort{\alpha}, \oshort{\alpha}, \beta, c, \kappa_1, \kappa_2)$ such that $f(Y) \le f(\tilde{X})$. Therefore, \eqref{eq:SufficientDecrease} follows from \eqref{eq:ERFDRmapERFDmapArmijoCondition} and \eqref{eq:ERFDRmapDelta}, which completes the proof for the case where $\ushort{r} = r$.

Let us now consider the case where $\ushort{r} < r$.
Since $f$ is continuous at $\ushort{X}$, there exists $\rho_0(\ushort{X}) \in (0,\infty)$ such that $f(\ball[\ushort{X}, \rho_0(\ushort{X})]) \subseteq [f(\ushort{X})-\delta(\ushort{X}), f(\ushort{X})+\delta(\ushort{X})]$. Define
\begin{equation*}
\varepsilon(\ushort{X}) \coloneq \min\{\Delta, {\textstyle\frac{1}{2}} \rho_0(\ushort{X}), {\textstyle\frac{1}{2}} \rho_1(\ushort{X})\}.
\end{equation*}
Let $X \in \ball[\ushort{X}, \varepsilon(\ushort{X})] \cap \R_{\le r}^{m \times n}$ and $Y \in \hyperref[algo:ERFDRmap]{\erfdr}(X; f, r, \ushort{\alpha}, \oshort{\alpha}, \beta, c, \kappa_1, \kappa_2, \Delta)$.
Let us first consider the case where $\rank X = r$. Then, by Proposition~\ref{prop:SingularValuesLipschitz},
\begin{equation*}
0 < \sigma_r(X) = \sigma_r(X) - \sigma_r(\ushort{X}) \le \norm{X-\ushort{X}} \le \varepsilon(\ushort{X}) \le \Delta.
\end{equation*}
Thus, there exist $\hat{X} \in \proj{\R_{r-1}^{m \times n}}{X}$ and $\tilde{X}^\mathrm{R} \in \hyperref[algo:ERFDmap]{\erfd}(\hat{X}; f, r, \ushort{\alpha}, \oshort{\alpha}, \beta, c, \kappa_1, \kappa_2)$ such that $f(Y) \le f(\tilde{X}^\mathrm{R})$. Moreover, $\hat{X} \in \ball[\ushort{X}, 2\varepsilon(\ushort{X})]$ since
\begin{equation*}
\norm{\hat{X}-\ushort{X}}
\le \norm{\hat{X}-X} + \norm{X-\ushort{X}}
\le \sigma_r(X) + \varepsilon(\ushort{X})
\le 2 \varepsilon(\ushort{X})
\le \min\{\rho_0(\ushort{X}), \rho_1(\ushort{X})\}.
\end{equation*}
As $X, \hat{X} \in \ball[\ushort{X}, \rho_0(\ushort{X})]$, it holds that
\begin{equation}
\label{eq:ERFDRmapCostContinuous}
f(\hat{X}) \le f(X) + 2\delta(\ushort{X}).
\end{equation}
As $X, \hat{X} \in \ball[\ushort{X}, \rho_1(\ushort{X})] \cap \R_{\le r}^{m \times n}$, it holds that
\begin{align*}
f(Y)
&\le f(\tilde{X}^\mathrm{R})\\
&\le f(\hat{X}) - c \, \kappa_1 \alpha_*(\ushort{X}) \s(\hat{X}; f, \R_{\le r}^{m \times n})^2\\
&\le f(\hat{X}) - c \, \kappa_1 \alpha_*(\ushort{X}) \frac{\norm{\nabla f(\hat{X})}^2}{\min\{m, n\}-r+1}\\
&\le f(\hat{X}) - 3 \delta(\ushort{X})\\
&\le f(X) - \delta(\ushort{X}),
\end{align*}
where the second inequality follows from \eqref{eq:ERFDRmapERFDmapArmijoCondition}, the third from the second inequality of~\eqref{eq:NormProjectionTangentConeDeterminantalVariety} since $\rank \hat{X} = r-1$, the fourth from \eqref{eq:ERFDRmapContinuityGradient} and \eqref{eq:ERFDRmapDelta}, and the fifth from \eqref{eq:ERFDRmapCostContinuous}.
Let us now consider the case where $\rank X < r$. There exists $\tilde{X} \in \hyperref[algo:ERFDmap]{\erfd}(X; f, r, \ushort{\alpha}, \oshort{\alpha}, \beta, c, \kappa_1, \kappa_2)$ such that $f(Y) \le f(\tilde{X})$. Therefore,
\begin{align*}
f(Y)
&\le f(\tilde{X})\\
&\le f(X) - c \, \kappa_1 \alpha_*(\ushort{X}) \s(X; f, \R_{\le r}^{m \times n})^2\\
&\le f(X) - c \, \kappa_1 \alpha_*(\ushort{X}) \frac{r-\rank X}{\min\{m, n\}-\rank X} \norm{\nabla f(X)}^2\\
&\le f(X) - c \, \kappa_1 \alpha_*(\ushort{X}) \frac{\norm{\nabla f(X)}^2}{\min\{m, n\}-r+1}\\
&\le f(X) - 3 \delta(\ushort{X})\\
&\le f(X) - \delta(\ushort{X}),
\end{align*}
where the second inequality follows from \eqref{eq:ERFDRmapERFDmapArmijoCondition}, the third from the second inequality of \eqref{eq:NormProjectionTangentConeDeterminantalVariety}, and the fifth from \eqref{eq:ERFDRmapContinuityGradient} and \eqref{eq:ERFDRmapDelta}.
\end{proof}

\begin{theorem}
\label{thm:ERFDR}
Consider a sequence generated by $\erfdr$ (Algorithm~\ref{algo:ERFDR}). If this sequence is finite, then its last element is B-stationary for~\eqref{eq:OptiDeterminantalVariety} in the sense of Definition~\ref{def:B-Stationarity}. If it is infinite, then all of its accumulation points, if any, are B-stationary for~\eqref{eq:OptiDeterminantalVariety}.
\end{theorem}

\begin{proof}
This follows from \cite[Proposition~3.2]{OlikierGallivanAbsil2024} and Lemma~\ref{lemma:ERFDRmap}.
\end{proof}

Theorem~\ref{thm:ERFDRconvergenceRate} offers a lower bound on the rate at which $\s(\cdot; f, \R_{\le r}^{m \times n})$ converges to zero along all bounded subsequences of every sequence generated by $\erfdr$. Importantly, Theorem~\ref{thm:ERFDR} is not a consequence of Theorem~\ref{thm:ERFDRconvergenceRate} since, as indicated in Section~\ref{subsec:ElementsVariationalGeometry}, the function $\s(\cdot; f, \R_{\le r}^{m \times n})$ can fail to be lower semicontinuous at every point of $\R_{< r}^{m \times n}$. The methods from \cite{OlikierUschmajewVandereycken} enjoy the same lower bound, but are not guaranteed to accumulate at B-stationary points of~\eqref{eq:OptiDeterminantalVariety}. Riemannian gradient descent~\cite[\S 2]{BoumalAbsilCartis} also enjoys this lower bound, but this fact does not imply Theorem~\ref{thm:ERFDRconvergenceRate}: $\R_{\le r}^{m \times n}$ is not a manifold and, while $\R_r^{m \times n}$ is a manifold, it is such that~\cite[Theorem~2.11]{BoumalAbsilCartis} is vacuous due to $\varrho = 0$.

\begin{theorem}
\label{thm:ERFDRconvergenceRate}
Let $(X_i)_{i \in \N}$ be a sequence generated by $\erfdr$ (Algorithm~\ref{algo:ERFDR}).
If a subsequence $(X_{i_k})_{k \in \N}$ is bounded, then, for all $k \in \N$,
\begin{equation}
\label{eq:ERFDRconvergenceRateSubsequence}
\sum_{l=0}^k \s(X_{i_l}; f, \R_{\le r}^{m \times n})^2 < \frac{f(X_0)-\inf_{l \in \N} f(X_{i_l})}{c \, \kappa_1 \min\left\{\ushort{\alpha}, \frac{2\beta\kappa_2(1-c)}{\lip_{\ball[X_0, \rho]}(\nabla f)}\right\}},
\end{equation}
where
\begin{equation}
\label{eq:ERFDRconvergenceRateSubsequenceBallLipschitz}
\rho \coloneq \sup_{l \in \N} \left(\norm{X_{i_l}-X_0} + \frac{\oshort{\alpha}}{\kappa_2} \norm{\nabla f(X_{i_l})}\right) < \infty,
\end{equation}
thus $\lim_{l \to \infty} \s(X_{i_l}; f, \R_{\le r}^{m \times n}) = 0$ and
\begin{equation}
\label{eq:ERFDRconvergenceRateSubsequenceBis}
\min_{l \in \{0, \dots, k\}} \s(X_{i_l}; f, \R_{\le r}^{m \times n}) < \frac{1}{\sqrt{k+1}} \sqrt{\frac{f(X_0)-\inf_{l \in \N} f(X_{i_l})}{c \, \kappa_1 \min\left\{\ushort{\alpha}, \frac{2\beta\kappa_2(1-c)}{\lip_{\ball[X_0, \rho]}(\nabla f)}\right\}}}.
\end{equation}
If, moreover, $(X_{i_k})_{k \in \N}$ converges to $X \in \R_{\le r}^{m \times n}$, then $\inf_{l \in \N} f(X_{i_l}) = f(X)$.
If $(X_i)_{i \in \N}$ is bounded, which is the case if the sublevel set $\{X \in \R_{\le r}^{m \times n} \mid f(X) \le f(X_0)\}$ is bounded, then all accumulation points have the same image by $f$ and, for all $\varepsilon \in (0, \infty)$,
\begin{equation}
\label{eq:ERFDRepsilonBstationaryPoint}
i_\varepsilon
\coloneq \min\{i \in \N \mid \s(X_i; f, \R_{\le r}^{m \times n}) \le \varepsilon\}
\le \left\lceil\frac{f(X_0)-\inf_{j \in \N}f(X_j)}{\varepsilon^2 c \, \kappa_1 \min\left\{\ushort{\alpha}, \frac{2\beta\kappa_2(1-c)}{\lip_{\ball[X_0, \rho]}(\nabla f)}\right\}}\right\rceil - 1.
\end{equation}
\end{theorem}

\begin{proof}
The statement about the image of the accumulation points is established in the proof of \cite[Corollary~6.3]{OlikierGallivanAbsil2024}.
The proof of the other statements follows that of \cite[(1.2.22)]{Nesterov2018}. Observe that $\norm{\nabla f(X_i)} > 0$ for all $i \in \N$.
Assume that a subsequence $(X_{i_k})_{k \in \N}$ is bounded. Let $\rho$ be as in \eqref{eq:ERFDRconvergenceRateSubsequenceBallLipschitz}. Then, for all $l \in \N$, it holds that $\ball[X_{i_l}, \frac{\oshort{\alpha}}{\kappa_2} \norm{\nabla f(X_{i_l})}] \subseteq \ball[X_0, \rho]$; indeed, for all $X \in \ball[X_{i_l}, \frac{\oshort{\alpha}}{\kappa_2} \norm{\nabla f(X_{i_l})}]$,
\begin{equation*}
\norm{X-X_0}
\le \norm{X-X_{i_l}} + \norm{X_{i_l}-X_0}
\le \frac{\oshort{\alpha}}{\kappa_2} \norm{\nabla f(X_{i_l})} + \norm{X_{i_l}-X_0}
\le \rho.
\end{equation*}
Thus, for all $l \in \N$,
\begin{align*}
f(X_{i_{l+1}})
&\le f(X_{i_l+1})\\
&\le f(\tilde{X}_{i_l})\\
&\le f(X_{i_l}) - c \, \kappa_1 \min\left\{\ushort{\alpha}, \frac{2\beta\kappa_2(1-c)}{\lip_{\ball[X_0, \rho]}(\nabla f)}\right\} \s(X_{i_l}; f, \R_{\le r}^{m \times n})^2,
\end{align*}
where the last inequality follows from \eqref{eq:ERFDmapCostDecrease}, and hence
\begin{equation*}
\s(X_{i_l}; f, \R_{\le r}^{m \times n})^2 \le \frac{f(X_{i_l})-f(X_{i_{l+1}})}{c \, \kappa_1 \min\left\{\ushort{\alpha}, \frac{2\beta\kappa_2(1-c)}{\lip_{\ball[X_0, \rho]}(\nabla f)}\right\}}.
\end{equation*}
For all $k \in \N$, summing the preceding inequality over $l \in \{0, \dots, k\}$ yields
\begin{align*}
(k+1) \min_{l \in \{0, \dots, k\}} \s(X_{i_l}; f, \R_{\le r}^{m \times n})^2
&\le \sum_{l=0}^k \s(X_{i_l}; f, \R_{\le r}^{m \times n})^2 \\
&\le \frac{f(X_{i_0})-f(X_{i_{k+1}})}{c \, \kappa_1 \min\left\{\ushort{\alpha}, \frac{2\beta\kappa_2(1-c)}{\lip_{\ball[X_0, \rho]}(\nabla f)}\right\}}\\
&< \frac{f(X_0)-\inf_{l \in \N} f(X_{i_l})}{c \, \kappa_1 \min\left\{\ushort{\alpha}, \frac{2\beta\kappa_2(1-c)}{\lip_{\ball[X_0, \rho]}(\nabla f)}\right\}},
\end{align*}
and \eqref{eq:ERFDRconvergenceRateSubsequence} and \eqref{eq:ERFDRconvergenceRateSubsequenceBis} follow. The ``moreover'' statement is clear.
Assume now that $(X_i)_{i \in \N}$ is bounded. By \eqref{eq:ERFDRconvergenceRateSubsequenceBis}, for all $\varepsilon \in (0, \infty)$ and $i \in \N$, if
\begin{equation*}
\frac{1}{\sqrt{i+1}} \sqrt{\frac{f(X_0)-\inf_{j \in \N}f(X_j)}{c \, \kappa_1 \min\left\{\ushort{\alpha}, \frac{2\beta\kappa_2(1-c)}{\lip_{\ball[X_0, \rho]}(\nabla f)}\right\}}} \le \varepsilon,
\end{equation*}
i.e.,
\begin{equation*}
i \ge \frac{f(X_0)-\inf_{j \in \N}f(X_j)}{\varepsilon^2 c \, \kappa_1 \min\left\{\ushort{\alpha}, \frac{2\beta\kappa_2(1-c)}{\lip_{\ball[X_0, \rho]}(\nabla f)}\right\}} - 1,
\end{equation*}
then there is $j \in \{0, \dots, i\}$ such that $\s(X_j; f, \R_{\le r}^{m \times n}) \le \varepsilon$, thus
\begin{equation*}
i_\varepsilon
\le j
\le \left\lceil\frac{f(X_0)-\inf_{j \in \N}f(X_j)}{\varepsilon^2 c \, \kappa_1 \min\left\{\ushort{\alpha}, \frac{2\beta\kappa_2(1-c)}{\lip_{\ball[X_0, \rho]}(\nabla f)}\right\}}\right\rceil - 1,
\end{equation*}
and \eqref{eq:ERFDRepsilonBstationaryPoint} is proven.
\end{proof}

The analysis conducted in Sections~\ref{sec:ERFDmap} and \ref{subsec:ConvergenceAnalysis} can be extended straightforwardly to the case where $f$ is only defined on an open subset of $\R^{m \times n}$ containing $\R_{\le r}^{m \times n}$.

\subsection{On the choice of $\Delta$}
\label{subsec:ChoiceDelta}
Theorem~\ref{thm:ERFDR} is valid for every sequence $(\Delta_i)_{i \in \N}$ bounded away from zero. However, $(\Delta_i)_{i \in \N}$ influences the sequence generated by $\erfdr$ since it governs the sensitivity of the rank reduction mechanism. Every action of the rank reduction mechanism increases both the exploration of the feasible set and the computational cost. The parameter $\Delta_i$ enables to tune the tradeoff between exploration and low computational cost at iteration~$i$: the smaller $\Delta_i$ the less likely the rank reduction mechanism takes action.

\section{A rank-increasing algorithm}
\label{sec:RankIncreasingAlgorithm}
As mentioned in Section~\ref{sec:Introduction}, optimization problems on the determinantal variety $\R_{\le r}^{m \times n}$ appear in several applications. In practice, the parameter $r$ influences both the computational work required by an optimization algorithm to produce a solution and the quality of the solution. However, the most suitable $r$, or even a suitable $r$, may not be known a priori. Rank-increasing and rank-adaptive algorithms aim at overcoming this difficulty. For example, the rank-increasing method proposed in \cite[Algorithm~1]{UschmajewVandereycken2014} starts with a low value of~$r$, which is increased progressively by a given constant. This allows the algorithm to avoid considering unnecessarily large ranks. Rank-adaptive algorithms are more sophisticated. For example, the Riemannian rank-adaptive method defined in \cite[Algorithm~3]{ZhouEtAl2016} increases or decreases the rank by an adaptively chosen amount as the iteration proceeds, offering a tunable tradeoff between the computational cost of producing a solution and the quality of the solution. $\ppgdr$ \cite[Algorithm~6.2]{OlikierGallivanAbsil2024} and $\erfdr$ (Algorithm~\ref{algo:ERFDR}) can be used to design rank-increasing schemes. An example based on $\erfdr$ is given as Algorithm~\ref{algo:RankIncreasingERFDR}.

\begin{algorithm}[H]
\caption{A rank-increasing algorithm based on $\erfdr$}
\label{algo:RankIncreasingERFDR}
\begin{algorithmic}[1]
\Require
$(f, r, \ushort{\alpha}, \oshort{\alpha}, \beta, c, \kappa_1, \kappa_2, \Delta, \tau, \varepsilon)$ where $f : \R^{m \times n} \to \R$ is differentiable with $\nabla f$ locally Lipschitz continuous, $r < \min\{m, n\}$ is a positive integer, $0 < \ushort{\alpha} \le \oshort{\alpha} < \infty$, $\beta, c, \tau \in (0,1)$, $\kappa_1 \in (0, \frac{1}{2}]$, $\kappa_2 \in (0, 1]$, and $\Delta, \varepsilon \in (0,\infty)$. 
\Input
$(X_0, r_0)$ such that $X_0 \in \R_{\le r_0}^{m \times n}$, $r_0 \in \{1, \dots, r\}$, and $\{X \in \R_{\le r}^{m \times n} \mid f(X) \le f(X_0)\}$ is bounded.
\Output
$((X_i)_{i=0}^{i=i_*}, (r_i)_{i=0}^{i=i_*})$ such that $i_* \in \N \cup \{\infty\}$ and, for all $i \in \{1, \dots, i_*\}$, $X_i \in \R_{\le r_{i-1}}^{m \times n}$, $r_i \in \{r_{i-1}, \dots, r\}$, and $\s(X_i; f, \R_{\le r_{i-1}}^{m \times n}) \le \tau^{i-1} \varepsilon$.

\State
$i \gets 0$;
\While
{$\s(X_i; f, \R_{\le r_i}^{m \times n}) > 0$}
\State
Iterate $\hyperref[algo:ERFDR]{\erfdr}$ on $(X_i; f, r_i, \ushort{\alpha}, \oshort{\alpha}, \beta, c, \kappa_1, \kappa_2, \Delta)$ at least once to compute $X_{i+1} \in \R_{\le r_i}^{m \times n}$ such that $\s(X_{i+1}; f, \R_{\le r_i}^{m \times n}) \le \tau^i\varepsilon$;
\label{algo:RankIncreasingERFDR:ERFDR}
\State
Choose $r_{i+1} \in \{r_i, \dots, r\}$;
\label{algo:RankIncreasingERFDR:RankIncrease}
\State
$i \gets i+1$;
\EndWhile
\State
$i_* \gets i$;
\end{algorithmic}
\end{algorithm}

By Theorems~\ref{thm:ERFDR} and~\ref{thm:ERFDRconvergenceRate}, line~\ref{algo:RankIncreasingERFDR:ERFDR} terminates; \eqref{eq:ERFDRepsilonBstationaryPoint} even provides an upper bound on the number of $\erfdr$ iterations required.
There is no constraint on the choice to be made in line~\ref{algo:RankIncreasingERFDR:RankIncrease}; the choice can be made, e.g., based on information depending on the application or by using a rank-increasing mechanism such as that in \cite[Algorithm~3]{ZhouEtAl2016}.

If $i_* < \infty$, then $\s(X_{i_*}; f, \R_{\le r_{i_*}}^{m \times n}) = 0$. Otherwise, since $(r_i)_{i \in \N}$ is in $\N$, monotonically nondecreasing, and upper bounded by $r$, there exists $\ushort{r} \in \{r_0, \dots, r\}$ such that $r_i = \ushort{r}$ for all $i \in \N$ large enough. Thus, Corollary~\ref{coro:RankIncreasingERFDR} readily follows from Theorems~\ref{thm:ERFDR} and \ref{thm:ERFDRconvergenceRate}.

\begin{corollary}
\label{coro:RankIncreasingERFDR}
Let $((X_i)_{i=0}^{i=i_*}, (r_i)_{i=0}^{i=i_*})$ be generated by Algorithm~\ref{algo:RankIncreasingERFDR}. If $i_* < \infty$, then $X_{i_*}$ is B-stationary on $\R_{\le r_{i_*}}^{m \times n}$, i.e., $\s(X_{i_*}; f, \R_{\le r_{i_*}}^{m \times n}) = 0$. Otherwise, letting $\ushort{r}$ denote the limit of $(r_i)_{i \in \N}$:
\begin{enumerate}
\item each of the accumulation points of $(X_i)_{i \in \N}$, of which there exists at least one, is B-stationary on $\R_{\le \ushort{r}}^{m \times n}$, i.e., is a zero of $\s(\cdot; f, \R_{\le \ushort{r}}^{m \times n})$;
\item $\lim_{i \to \infty} \s(X_i; f, \R_{\le \ushort{r}}^{m \times n}) = 0$;
\item all accumulation points of $(X_i)_{i \in \N}$ have the same image by $f$.
\end{enumerate}
\end{corollary}

Corollary~\ref{coro:RankIncreasingERFDR} is stronger than \cite[Theorem~2]{ZhouEtAl2016}. Furthermore, the latter relies on \cite[Assumption~6]{ZhouEtAl2016}, which can fail to be satisfied even by polynomial functions \cite[Proposition~7.21]{OlikierGallivanAbsil2023}.

$\erfdr$ can be replaced with $\ppgdr$ in Algorithm~\ref{algo:RankIncreasingERFDR}. By \cite[Theorem~6.2 and Corollary~6.3]{OlikierGallivanAbsil2024}, a result analogous to Corollary~\ref{coro:RankIncreasingERFDR} holds.

\section{The $\crfd$ map and the $\crfdr$ algorithm}
\label{sec:CRFDmapCRFDR}
The $\crfd$ map is introduced in Definition~\ref{def:CRFD(R)map} as a subclass of the $\erfd$ map (Algorithm~\ref{algo:ERFDmap}) based on Propositions~\ref{prop:ProjectionOntoClosedConeMinUnitSphere} and \ref{prop:CRFD}. It exploits the idea raised at the end of Section~\ref{subsec:ElementsVariationalGeometry} of finding a computationally cheaper surrogate for~\eqref{eq:ProjectionRestrictedTangentConeDeterminantalVariety} when the rank of the input is smaller than $r$. The $\crfdr$ map, which is the iteration map of $\crfdr$ and uses the $\crfd$ map as a subroutine, is also introduced in Definition~\ref{def:CRFD(R)map}.

\begin{proposition}
\label{prop:ProjectionOntoClosedConeMinUnitSphere}
For every closed cone $\mathcal{C} \subseteq \R^{m \times n}$,
\begin{equation*}
\inf_{X \in \R^{m \times n} \setminus \{0_{m \times n}\}} \frac{\norm{\proj{\mathcal{C}}{X}}}{\norm{X}}
= \min_{X \in \partial \ball(0_{m \times n}, 1)} \norm{\proj{\mathcal{C}}{X}},
\end{equation*}
and this quantity is positive if and only if $\mathcal{C}^* = \{0_{m \times n}\}$.
\end{proposition}

\begin{proof}
Let $\mathcal{C} \subseteq \R^{m \times n}$ be a closed cone.
For all $X \in \R^{m \times n} \setminus \{0_{m \times n}\}$ and $\lambda \in (0, \infty)$,
\begin{equation*}
\dist(\lambda X, \mathcal{C})
= \inf_{Y \in \mathcal{C}} \norm{\lambda X-Y}
= \lambda \inf_{Y \in \mathcal{C}} \norm{X-\frac{Y}{\lambda}}
= \lambda \inf_{Y \in \mathcal{C}} \norm{X-Y}
= \lambda \, \dist(X, \mathcal{C}).
\end{equation*}
Thus, by Proposition~\ref{prop:ProjectionOntoClosedCone}, for all $X \in \R^{m \times n} \setminus \{0_{m \times n}\}$,
\begin{align*}
\frac{\norm{\proj{\mathcal{C}}{X}}}{\norm{X}}
&= \frac{\sqrt{\norm{X}^2-\dist(X, \mathcal{C})^2}}{\norm{X}}\\
&= \sqrt{1-\left(\frac{\dist(X, \mathcal{C})}{\norm{X}}\right)^2}\\
&= \sqrt{1-\dist(\frac{X}{\norm{X}}, \mathcal{C})^2}\\
&= \norm{\proj{\mathcal{C}}{\frac{X}{\norm{X}}}}.
\end{align*}
Hence,
\begin{equation*}
\inf_{X \in \R^{m \times n} \setminus \{0_{m \times n}\}} \frac{\norm{\proj{\mathcal{C}}{X}}}{\norm{X}} = \inf_{X \in \partial \ball(0_{m \times n}, 1)} \norm{\proj{\mathcal{C}}{X}}.
\end{equation*}
The equality follows from the extreme value theorem. The ``if and only if'' statement follows from Proposition~\ref{prop:ProjectionOntoClosedCone}.
\end{proof}

\begin{proposition}
\label{prop:CRFD}
Let $X \in \R_{< r}^{m \times n}$ and $\ushort{r} \coloneq \rank X$. Let $\mathcal{C} \subseteq \R_{\le r-\ushort{r}}^{m \times n}$ be a closed cone such that $\mathcal{C}^* = \{0_{m \times n}\}$. Then, every $G \in \proj{\mathcal{C}}{-\nabla f(X)}$ satisfies \eqref{eq:ERFDmapSearchDirection} with $\kappa_1 = \min\limits_{Y \in \partial \ball(0_{m \times n}, 1)} \norm{\proj{\mathcal{C}}{Y}}^2$ and $\kappa_2 = 1$ and $X+\alpha G \in \R_{\le r}^{m \times n}$ for all $\alpha \in (0, \infty)$.
\end{proposition}

\begin{proof}
Let $Z \coloneq -\nabla f(X)$ and $G \in \proj{\mathcal{C}}{Z}$. By Propositions~\ref{prop:ProjectionOntoClosedCone} and \ref{prop:ProjectionOntoClosedConeMinUnitSphere} and the first inequality of~\eqref{eq:NormProjectionTangentConeDeterminantalVariety},
\begin{align*}
\ip{Z}{G}
&= \norm{G}^2\\
&= \norm{\proj{\mathcal{C}}{Z}}^2\\
&\ge \min_{Y \in \partial \ball(0_{m \times n}, 1)} \norm{\proj{\mathcal{C}}{Y}}^2 \norm{Z}^2\\
&\ge \min_{Y \in \partial \ball(0_{m \times n}, 1)} \norm{\proj{\mathcal{C}}{Y}}^2 \s(X; f, \R_{\le r}^{m \times n})^2,
\end{align*}
which establishes the first statement. The second statement holds by subadditivity of the rank.
\end{proof}

\begin{definition}
\label{def:CRFD(R)map}
The $\erfd$ map (Algorithm~\ref{algo:ERFDmap}) is called the $\crfd$ map if, given an input $X \in \R_{\le r}^{m \times n}$ of rank denoted by $\ushort{r}$, the element of $\R^{m \times n}$ to be chosen in line~\ref{algo:ERFDmap:SearchDirection} is selected in $\proj{\restancone{\R_{\le r}^{m \times n}}{X}}{-\nabla f(X)}$ if $\ushort{r} = r$ and in $\proj{\mathcal{C}}{-\nabla f(X)}$ if $\ushort{r} < r$, where $\mathcal{C} \subseteq \R_{\le r-\ushort{r}}^{m \times n}$ is a closed cone satisfying the following properties:
\begin{enumerate}
\item $\min_{Y \in \partial \ball(0_{m \times n}, 1)} \norm{\proj{\mathcal{C}}{Y}}^2 \ge \kappa_1$;
\item computing a projection onto $\mathcal{C}$ and an SVD of this projection is at most as expensive as computing the Frobenius norm of an element of $\R^{m \times n}$.
\end{enumerate}
The $\erfdr$ map is called the $\crfdr$ map if it uses the $\crfd$ map in lines~\ref{algo:ERFDRmap:ERFD} and~\ref{algo:ERFDRmap:ERFD_RankReduction}.
\end{definition}

Examples of a closed cone $\mathcal{C} \subseteq \R_{\le 1}^{m \times n}$ satisfying the requirements of Definition~\ref{def:CRFD(R)map} for all $X \in \R_{< r}^{m \times n}$ are given in Table~\ref{tab:ClosedConeExamples}.

\begin{table}[h]
\begin{center}
\begin{spacing}{1.6}
\begin{tabular}{ll}
\hline
$\mathcal{C}$ & $\min\limits_{X \in \partial \ball(0_{m \times n}, 1)} \norm{\proj{\mathcal{C}}{X}}^2$\\
\hline
all elements of $\R^{m \times n}$ with at most one nonzero entry & $1/mn$\\
\hline
all elements of $\R^{m \times n}$ with at most one nonzero row & $1/m$\\
\hline
all elements of $\R^{m \times n}$ with at most one nonzero column & $1/n$\\
\hline
\end{tabular}
\end{spacing}
\end{center}
\caption{Examples of a closed cone $\mathcal{C} \subseteq \R_{\le 1}^{m \times n}$ satisfying the requirements of Definition~\ref{def:CRFD(R)map} for all $X \in \R_{< r}^{m \times n}$.}
\label{tab:ClosedConeExamples}
\end{table}

Proposition~\ref{prop:CRFD} still holds if ``$\proj{\mathcal{C}}{-\nabla f(X)}$'' and ``$\kappa_2 = 1$'' are replaced with ``$\proj{\restancone{\R_{\ushort{r}}^{m \times n}}{X}}{-\nabla f(X)} + \proj{\mathcal{C}}{-\nabla f(X)}$'' and ``$\kappa_2 = \frac{1}{2}$'', respectively. This defines a potentially interesting alternative to the $\crfd$ map.

\section{Practical implementation and computational cost of the $\crfdr$ map}
\label{sec:PracticalImplementationComputationalCostCRFDRmap}
Following \cite[\S 7]{OlikierAbsil2023}, this section compares the computational cost of the $\crfdr$ map (Definition~\ref{def:CRFD(R)map}) with that of the $\rfdr$ map \cite[Algorithm~2]{OlikierAbsil2023} (Section~\ref{subsec:ComparisonWithRFDRmap}) and that of the $\rfd$ map \cite[Algorithm~1]{OlikierAbsil2023} (Section~\ref{subsec:ComparisonWithRFDmap}). The comparison is based on detailed implementations of these algorithms involving only evaluations of $f$ and $\nabla f$, a projection onto a closed cone $\mathcal{C}$ as in Definition~\ref{def:CRFD(R)map}, an SVD of this projection, and some operations from linear algebra:
\begin{enumerate}
\item matrix multiplication requiring at most $2mnr$ flops;
\item QR factorization with column pivoting \cite[Algorithm~5.4.1]{GolubVanLoan} of a matrix with $m$ or $n$ rows and at most $r$ columns;
\item \emph{small-scale SVD}, i.e., the matrix to decompose is $m$-by-$r$, $r$-by-$n$, or its dimensions are at most~$r$;
\item \emph{large-scale SVD}, i.e., truncated SVD of rank at most $r-\ushort{r}$ of an $m$-by-$n$ matrix whose rank can be as large as $\min\{m, n\}-\ushort{r}$, with $\ushort{r} \in \{0, \dots, r-1\}$.
\end{enumerate}
The SVD terminology is that from Section~\ref{subsec:BasicFactsAboutSingularValues}. In this list, only the (truncated) SVD relies on a convergent algorithm rather than an algorithm with finite complexity such as those for QR factorization. Furthermore, if $\mathcal{C}$ is chosen among the examples given in Table~\ref{tab:ClosedConeExamples}, then projecting onto $\mathcal{C}$ amounts to find, in an $m$-by-$n$ matrix, the entry of maximal absolute value or the row or column of maximal Frobenius norm, and an SVD of the projection is readily available.

\subsection{Detailed implementation of the $\crfdr$ map}
\label{subsec:DetailedImplementationCRFDRmap}
Detailed implementations of the $\rfdr$ and $\crfdr$ maps are given in \cite[Algorithms~8]{OlikierAbsil2023} and Algorithm~\ref{algo:DetailedCRFDRmap}, respectively.\footnote{Matlab implementations of $\rfdr$ and $\crfdr$ based on these detailed implementations are available at \url{https://github.com/golikier/BouligandStationarityLowRankOptimization}.} Algorithm~\ref{algo:DetailedCRFDRmap} uses as a subroutine Algorithm~\ref{algo:DetailedCRFDmapSVDs}, which is a detailed implementation of the $\crfd$ map. Those algorithms work with factorizations of the involved matrices as much as possible.

Algorithm~\ref{algo:DetailedCRFDmapSVDs} involves one evaluation of $\nabla f$ (in line~\ref{algo:DetailedCRFDmapSVDs:GradientEvaluation}), at most
\begin{equation}
\label{eq:MaxNumIterationsDetailedCRFDmapSVDs}
1+\max\left\{0, \left\lceil\ln\left(\frac{2(1-c)}{\alpha\lip_{\ball[X, \alpha\norm{\nabla f(X)}]}(\nabla f)}\right)/\ln(\beta)\right\rceil\right\}
\end{equation}
evaluations of $f$ (in the while loop that is executed), at most one projection onto $\mathcal{C}$ and an SVD of this projection (in line~\ref{algo:DetailedCRFDmapSVDs:ProjectionClosedCone}), zero or two QR factorization(s) with column pivoting of a matrix with $m$ or $n$ rows and at most $r$ columns (in line~\ref{algo:DetailedCRFDmapSVDs:QRfactorizations}), matrix multiplications requiring at most $2mnr$ flops \cite[Table~1.1.2]{GolubVanLoan}, and one small-scale SVD (just after the while loop that is executed). The upper bound \eqref{eq:MaxNumIterationsDetailedCRFDmapSVDs} follows from \eqref{eq:ERFDmapMaxNumIterationsWhile}.

If the input is $0_{m \times n}$, then the QR factorizations in line~\ref{algo:DetailedCRFDmapSVDs:QRfactorizations} and the small-scale SVD in line~\ref{algo:DetailedCRFDmapSVDs:NonMaximalRankSmallScaleSVD} are not needed.

The input and output of Algorithm~\ref{algo:DetailedCRFDmapSVDs} are SVDs because it is a subroutine of Algorithm~\ref{algo:DetailedCRFDRmap}, which requires an SVD as input.

\begin{algorithm}[H]
\caption{Detailed $\crfd$ map via SVDs}
\label{algo:DetailedCRFDmapSVDs}
\begin{algorithmic}[1]
\Require
$(f, r, \alpha, \beta, c, \kappa_1)$ where $f : \R^{m \times n} \to \R$ is differentiable with $\nabla f$ locally Lipschitz continuous, $r < \min\{m, n\}$ is a positive integer, $\alpha \in (0, \infty)$, $\beta, c \in (0, 1)$, and $\kappa_1 \in (0, \frac{1}{2}]$.
\Input
$(U, \Sigma, V)$ where $U \Sigma V^\tp \in \R_{\ushort{r}}^{m \times n}$ is an SVD and $\ushort{r} \in \{0, \dots, r\}$.
\Output
$(\tilde{U}, \tilde{\Sigma}, \tilde{V})$ where $\tilde{U}\tilde{\Sigma}\tilde{V}^\tp \in \R_{\tilde{r}}^{m \times n}$ is an SVD, $\tilde{r} \in \{0, \dots, r\}$, and $\tilde{U}\tilde{\Sigma}\tilde{V}^\tp \in \hyperref[algo:ERFDmap]{\erfd}(U \Sigma V^\tp; f, r, \alpha, \alpha, \beta, c, \kappa_1, 1)$.

\State
$G \gets -\nabla f(U \Sigma V^\tp)$;
\label{algo:DetailedCRFDmapSVDs:GradientEvaluation}
\If
{$\ushort{r} = r$}
\State
$G_1 \gets U^\tp G$;
$G_2 \gets G V$;
$s_1 \gets \norm{G_1}^2$;
$s_2 \gets \norm{G_2}^2$;
\label{algo:DetailedCRFDmapSVDs:MaximalRank}
\If
{$s_1 \ge s_2$}
\State
$X_2 \gets \Sigma V^\tp$;
\While
{$f(U(X_2+\alpha G_1)) > f(U \Sigma V^\tp) - c \, \alpha \, s_1$}
\Comment{$UG_1 \in \proj{\restancone{\R_{\le r}^{m \times n}}{U \Sigma V^\tp}}{G}$.}
\State
$\alpha \gets \alpha \beta$;
\EndWhile
\State
Compute an SVD $\tilde{U} \tilde{\Sigma} \tilde{V}^\tp \in \R_{\tilde{r}}^{r \times n}$ of $X_2+\alpha G_1$;
\State
$\tilde{U} \gets U \tilde{U}$;
\Else
\State
$X_1 \gets U \Sigma$;
\While
{$f((X_1+\alpha G_2)V^\tp) > f(U \Sigma V^\tp) - c \, \alpha \, s_2$}
\Comment{$G_2V^\tp \in \proj{\restancone{\R_{\le r}^{m \times n}}{U \Sigma V^\tp}}{G}$.}
\State
$\alpha \gets \alpha \beta$;
\EndWhile
\State
Compute an SVD $\tilde{U} \tilde{\Sigma} \tilde{V}^\tp \in \R_{\tilde{r}}^{m \times r}$ of $X_1+\alpha G_2$;
\State
$\tilde{V} \gets V \tilde{V}$;
\EndIf
\label{algo:DetailedCRFDmapSVDs:MaximalRankEnd}
\Else
\State
Choose, e.g., in Table~\ref{tab:ClosedConeExamples}, a closed cone $\mathcal{C} \subseteq \R_{\le r-\ushort{r}}^{m \times n}$ satisfying the requirements of Definition~\ref{def:CRFD(R)map};
\label{algo:DetailedCRFDmapSVDs:NonMaximalRank}
\State
Choose $\oshort{G} \in \proj{\mathcal{C}}{G}$ and write it via SVD as $\oshort{G} = \oshort{U} \oshort{\Sigma} \oshort{V}^\tp \in \R_{\oshort{r}}^{m \times n}$;
\Comment{Recall that $1 \le \oshort{r} \le r-\ushort{r}$.}
\label{algo:DetailedCRFDmapSVDs:ProjectionClosedCone}
\State
$s \gets \norm{\oshort{\Sigma}}^2$;
\State
Compute QR factorizations with column pivoting $[U \; \oshort{U}] = \hat{U} R_1$ and $[V \; \oshort{V}] = \hat{V} R_2$ where $\hat{U} \in \st(r_1, m)$, $\hat{V} \in \st(r_2, n)$, and $r_1, r_2 \in \{\max\{\ushort{r}, \oshort{r}\}, \dots, \ushort{r}+\oshort{r}\}$;
\label{algo:DetailedCRFDmapSVDs:QRfactorizations}
\While
{$f(\hat{U} R_1 \diag(\Sigma, \alpha\oshort{\Sigma}) R_2^\tp \hat{V}^\tp) > f(U \Sigma V^\tp) - c \, \alpha \, s$}
\State
$\alpha \gets \alpha \beta$;
\EndWhile
\State
Compute an SVD $\tilde{U} \tilde{\Sigma} \tilde{V}^\tp \in \R_{\tilde{r}}^{r_1 \times r_2}$ of $R_1 \diag(\Sigma, \alpha\oshort{\Sigma}) R_2^\tp$;
\label{algo:DetailedCRFDmapSVDs:NonMaximalRankSmallScaleSVD}
\State
$\tilde{U} \gets \hat{U} \tilde{U}$;
$\tilde{V} \gets \hat{V} \tilde{V}$;
\label{algo:DetailedCRFDmapSVDs:NonMaximalRankEnd}
\EndIf
\State
Return $(\tilde{U}, \tilde{\Sigma}, \tilde{V})$.
\end{algorithmic}
\end{algorithm}

In the worst case, i.e., if the input has rank $r$ and its smallest singular value is smaller than or equal to $\Delta$, then Algorithm~\ref{algo:DetailedCRFDRmap} calls Algorithm~\ref{algo:DetailedCRFDmapSVDs} twice. In the first call, applied to the input, lines~\ref{algo:DetailedCRFDmapSVDs:MaximalRank} to \ref{algo:DetailedCRFDmapSVDs:MaximalRankEnd} are executed. In the second call, applied to a projection of the input onto $\R_{r-1}^{m \times n}$, it holds that $\ushort{r} = r-1$, hence lines~\ref{algo:DetailedCRFDmapSVDs:NonMaximalRank} to \ref{algo:DetailedCRFDmapSVDs:NonMaximalRankEnd} are executed. If $r > 1$, then, besides matrix multiplications, the two calls require two evaluations of $\nabla f$, at most two times \eqref{eq:MaxNumIterationsDetailedCRFDmapSVDs} evaluations of $f$, one projection onto the closed cone $\mathcal{C}$, an SVD of this projection, two QR factorization(s) with column pivoting of a matrix with $m$ or $n$ rows and $r$ columns, and two small-scale SVDs.

\begin{algorithm}[H]
\caption{Detailed $\crfdr$ map}
\label{algo:DetailedCRFDRmap}
\begin{algorithmic}[1]
\Require
$(f, r, \alpha, \beta, c, \kappa_1, \Delta)$ where $f : \R^{m \times n} \to \R$ is differentiable with $\nabla f$ locally Lipschitz continuous, $r < \min\{m, n\}$ is a positive integer, $\alpha \in (0, \infty)$, $\beta, c \in (0, 1)$, $\kappa_1 \in (0, \frac{1}{2}]$, and $\Delta \in (0, \infty)$.
\Input
$(U, \Sigma, V)$ where $U \Sigma V^\tp \in \R_{\ushort{r}}^{m \times n}$ is an SVD and $\ushort{r} \in \{0, \dots, r\}$.
\Output
$(\tilde{U}, \tilde{\Sigma}, \tilde{V})$ where $\tilde{U}\tilde{\Sigma}\tilde{V}^\tp \in \R_{\tilde{r}}^{m \times n}$ is an SVD, $\tilde{r} \in \{0, \dots, r\}$, and $\tilde{U}\tilde{\Sigma}\tilde{V}^\tp \in \hyperref[algo:ERFDRmap]{\erfdr}(U \Sigma V^\tp; f, r, \alpha, \alpha, \beta, c, \kappa_1, 1, \Delta)$.

\State
$(\tilde{U}, \tilde{\Sigma}, \tilde{V}) \gets \text{Algorithm~\ref{algo:DetailedCRFDmapSVDs}}(U, \Sigma, V; f, r, \alpha, \beta, c, \kappa_1)$;
\If
{$\ushort{r} = r$ and $\Sigma(r, r) \le \Delta$}
\State
$(\tilde{U}^\mathrm{R}, \tilde{\Sigma}^\mathrm{R}, \tilde{V}^\mathrm{R}) \gets \text{Algorithm~\ref{algo:DetailedCRFDmapSVDs}}(U(\mathord{:}, 1\mathord{:}r-1), \Sigma(1\mathord{:}r-1, 1\mathord{:}r-1), V(\mathord{:}, 1\mathord{:}r-1); f, r, \alpha, \beta, c, \kappa_1)$;
\If
{$f(\tilde{U}^\mathrm{R} \tilde{\Sigma}^\mathrm{R} (\tilde{V}^\mathrm{R})^\tp) < f(\tilde{U} \tilde{\Sigma} \tilde{V}^\tp)$}
\State
$(\tilde{U}, \tilde{\Sigma}, \tilde{V}) \gets (\tilde{U}^\mathrm{R}, \tilde{\Sigma}^\mathrm{R}, \tilde{V}^\mathrm{R})$;
\EndIf
\EndIf
\State
Return $(\tilde{U}, \tilde{\Sigma}, \tilde{V})$.
\end{algorithmic}
\end{algorithm}

\subsection{Comparison with the $\rfdr$ map}
\label{subsec:ComparisonWithRFDRmap}
Table~\ref{tab:CostRFDRvsCRFDR} summarizes the operations required by \cite[Algorithm~8]{OlikierAbsil2023} and Algorithm~\ref{algo:DetailedCRFDRmap}, matrix multiplication excluded, in the case where $r > 1$, the input has rank $r$, and its smallest singular value is smaller than or equal to~$\Delta$. The only difference is that, where \cite[Algorithm~8]{OlikierAbsil2023} requires a large-scale SVD, Algorithm~\ref{algo:DetailedCRFDRmap} merely requires a projection onto $\mathcal{C}$, an SVD of this projection, and two QR factorizations with column pivoting of a matrix with $m$ or $n$ rows and $r$ columns.

\begin{table}[h]
\begin{center}
\begin{tabular}{*{3}{l}}
\hline
Operation & \cite[Algorithm~8]{OlikierAbsil2023} ($\rfdr$) & Algorithm~\ref{algo:DetailedCRFDRmap} ($\crfdr$)\\
\hline\hline
evaluation of $f$ & at most $2 \cdot \eqref{eq:MaxNumIterationsDetailedCRFDmapSVDs}$ & at most $2 \cdot \eqref{eq:MaxNumIterationsDetailedCRFDmapSVDs}$ \\
\hline
evaluation of $\nabla f$ & $2$ & $2$ \\
\hline
QR & $0$ & $2$ \\
\hline
small-scale SVD & $2$ & $2$ \\
\hline
large-scale SVD & $1$ & $0$ \\
\hline
SVD-projection onto $\mathcal{C}$ & $0$ & $1$ \\
\hline
\end{tabular}
\end{center}
\caption{Operations required by \cite[Algorithm~8]{OlikierAbsil2023} and Algorithm~\ref{algo:DetailedCRFDRmap}, matrix multiplication excluded, in the case where $r > 1$, the input has rank $r$, and its smallest singular value is smaller than or equal to $\Delta$. ``QR'' means QR factorization with column pivoting of a matrix with $m$ or $n$ rows and $r$ columns. ``SVD-projection onto $\mathcal{C}$'' means computing a projection onto $\mathcal{C}$ and writing it as an SVD, which is straightforward if $\mathcal{C}$ is chosen from Table~\ref{tab:ClosedConeExamples}.}
\label{tab:CostRFDRvsCRFDR}
\end{table}

\subsection{Comparison with the $\rfd$ map}
\label{subsec:ComparisonWithRFDmap}
$\rfd$ \cite[Algorithm~4]{SchneiderUschmajew2015} does not satisfy the property that accumulation points are B-stationary for~\eqref{eq:OptiDeterminantalVariety} \cite[\S 8]{OlikierAbsil2023} whereas, as shown above, CRFDR does. This comes at a computational cost which, however, is the least amongst all methods that satisfy the property. In the typical case where the sequence generated by $\rfd$ is contained in $\R_r^{m \times n}$, this computational overhead is twofold. First, the iterates of $\crfdr$ are stored as SVDs, while those of $\rfd$ merely need to be stored in a format that reveals the rank and orthonormal bases of the row and column spaces. Second, given an input $X \in \R_r^{m \times n}$ such that $\sigma_r(X) \le \Delta$, the $\crfdr$ map applies the $\crfd$ map to $\hat{X} \in \proj{\R_{r-1}^{m \times n}}{X}$; $\hat{X}$ is readily avaible since the iterates are stored as SVDs. This involves the following: (i) computing an arbitrary projection of $-\nabla f(\hat{X})$ onto a closed cone as in Definition~\ref{def:CRFD(R)map} and writing it as an SVD, which requires no computational overhead beyond evaluating $-\nabla f(\hat{X})$ if $\mathcal{C}$ is the first example given in Table~\ref{tab:ClosedConeExamples}; (ii) at most \eqref{eq:MaxNumIterationsDetailedCRFDmapSVDs} evaluations of $f$; (iii) two QR factorizations with column pivoting of a matrix with $m$ or $n$ rows and $r$ columns; (iv) one SVD of a matrix whose dimensions are at most $r$.

\section{Conclusion}
\label{sec:Conclusion}
$\pgd$, $\ppgdr$, $\rfdr$, $\hrtr$, and some of their hybridizations are the only algorithms in the literature guaranteed to generate a sequence in $\R_{\le r}^{m \times n}$ whose accumulation points are B-stationary for~\eqref{eq:OptiDeterminantalVariety} in the sense of Definition~\ref{def:B-Stationarity}. All these algorithms can require SVDs or eigenvalue decompositions that are computationally prohibitive in the typical case where $r \ll \min\{m, n\}$ if $\nabla f$ does not have low rank. Among them, $\rfdr$ is the most parsimonious with these computationally expensive operations.

The main contribution of this paper is $\crfdr$ (Definition~\ref{def:CRFD(R)map}), a modification of $\rfdr$ that generates a sequence in $\R_{\le r}^{m \times n}$ whose accumulation points are B-stationary for~\eqref{eq:OptiDeterminantalVariety}, while requiring none of the computationally expensive operations mentioned in the preceding paragraph (Table~\ref{tab:CostRFDRvsCRFDR}). In the class of algorithms generating a sequence in $\R_{\le r}^{m \times n}$ whose accumulation points are provably B-stationary for~\eqref{eq:OptiDeterminantalVariety}, $\crfdr$ is that with the smallest computational overhead compared with $\rfd$ \cite[Algorithm~4]{SchneiderUschmajew2015}, which is not in that class \cite[\S 8]{OlikierAbsil2023}; see Section~\ref{subsec:ComparisonWithRFDmap}. The measure of B-stationarity defined in~\eqref{eq:NormProjectionNegativeGradientOntoTangentConeDeterminantalVariety} converges to zero along bounded subsequences at a rate at least $O(1/\sqrt{i+1})$, where $i$ is the iteration counter. A side contribution of this paper is a rank-increasing scheme (Algorithm~\ref{algo:RankIncreasingERFDR}), which can be of interest if the parameter $r$ in~\eqref{eq:OptiDeterminantalVariety} is potentially overestimated. A promising perspective is to use $\crfdr$ to design hybrid algorithms based on the framework presented in \cite[\S 3]{OlikierGallivanAbsil2024}.


\bibliographystyle{siamplain}
\bibliography{CRFDR_bib}
\end{document}

%% file: CRFDR_shared.tex
\usepackage{amsfonts,amssymb}
\usepackage{mathtools}
\usepackage{graphicx}
\usepackage{setspace}
\usepackage{hyperref}
\usepackage{color}
\usepackage{algorithm}
\usepackage{algorithmicx}
\usepackage{algpseudocode}
\algnewcommand\algorithmicinput{\textbf{Input:}}
\algnewcommand\Input{\item[\algorithmicinput]}
\algnewcommand\algorithmicoutput{\textbf{Output:}}
\algnewcommand\Output{\item[\algorithmicoutput]}
\usepackage{chngcntr}
\counterwithin{table}{section}

\ifpdf
  \DeclareGraphicsExtensions{.eps,.pdf,.png,.jpg}
\else
  \DeclareGraphicsExtensions{.eps}
\fi

\newcommand{\N}{\mathbb{N}}
\newcommand{\R}{\mathbb{R}}

\newcommand{\dist}{d}
\newcommand{\ball}{B}
\newcommand{\norm}[1]{\Vert #1 \Vert}
\newcommand{\ip}[2]{\langle #1, #2 \rangle}
\DeclareMathOperator*{\argmin}{argmin}
\DeclareMathOperator*{\lip}{Lip}
\DeclareMathOperator{\s}{s}

\newcommand{\proj}[2]{P_{#1}(#2)}

\makeatletter
\def\widebreve{\mathpalette\wide@breve}
\def\wide@breve#1#2{\sbox\z@{$#1#2$}%
     \mathop{\vbox{\m@th\ialign{##\crcr
\kern0.08em\brevefill#1{0.8\wd\z@}\crcr\noalign{\nointerlineskip}%
                    $\hss#1#2\hss$\crcr}}}\limits}
\def\brevefill#1#2{$\m@th\sbox\tw@{$#1($}%
  \hss\resizebox{#2}{\wd\tw@}{\rotatebox[origin=c]{90}{\upshape(}}\hss$}
\makeatletter

\newcommand{\gencone}[4]{{#1}_{#2}^{#4}(#3)} 
\newcommand{\tancone}[2]{\gencone{T}{#1}{#2}{}} 
\newcommand{\restancone}[2]{\gencone{\widebreve{T}}{#1}{#2}{}} 
\newcommand{\norcone}[2]{\gencone{N}{#1}{#2}{}} 

\newcommand{\oshort}[1]{\mkern 0.9mu\overline{\mkern-0.9mu#1\mkern-0.9mu}\mkern 0.9mu}
\newcommand{\ushort}[1]{\mkern 0.9mu\underline{\mkern-0.9mu#1\mkern-0.9mu}\mkern 0.9mu}

\DeclareMathOperator{\tr}{tr}
\newcommand{\tp}{\top}
\newcommand{\mpinv}{\dagger}
\DeclareMathOperator{\rank}{rank}
\newcommand{\st}{\mathrm{St}}

\newcommand{\pgd}{\mathrm{PGD}} 
\newcommand{\ppgdr}{\mathrm{P}^2\mathrm{GDR}}
\newcommand{\rfd}{\mathrm{RFD}} 
\newcommand{\rfdr}{\mathrm{RFDR}}
\newcommand{\crfd}{\mathrm{CRFD}} 
\newcommand{\crfdr}{\mathrm{CRFDR}} 
\newcommand{\erfd}{\mathrm{ERFD}} 
\newcommand{\erfdr}{\mathrm{ERFDR}} 
\newcommand{\hrtr}{\mathrm{HRTR}} 

\newsiamremark{remark}{Remark}
\newsiamremark{hypothesis}{Hypothesis}
\crefname{hypothesis}{Hypothesis}{Hypotheses}
\newsiamthm{claim}{Claim}

\headers{B-stationary points in low-rank optimization}{G. Olikier and P.-A. Absil}

\title{Computing Bouligand stationary points efficiently in low-rank optimization\thanks{\today
\funding{This work was funded by the Fonds de la Recherche Scientifique -- FNRS and the Fonds Wetenschappelijk Onderzoek -- Vlaanderen under EOS Project no 30468160, by the Fonds de la Recherche Scientifique -- FNRS under Grant no T.0001.23, by the ERC grant \#786854 G-Statistics from the European Research Council under the European Union's Horizon 2020 research and innovation program, and by the French government through the 3IA Côte d'Azur Investments ANR-19-P3IA-0002 managed by the National Research Agency.}}}

\author{Guillaume Olikier\thanks{Université Côte d'Azur and Inria, Epione Project Team, 2004 route des Lucioles - BP 93, 06902 Sophia Antipolis Cedex, France
  (\email{guillaume.olikier@inria.fr}).}
\and P.-A. Absil\thanks{ICTEAM Institute, UCLouvain, Avenue Georges Lema\^{\i}tre 4, 1348 Louvain-la-Neuve, Belgium
  (\email{pa.absil@uclouvain.be}).}
}

\usepackage{amsopn}
\DeclareMathOperator{\diag}{diag}
